\documentclass[preprint]{elsarticle}
\usepackage[utf8x]{inputenc}
\usepackage[OT1]{fontenc}
\usepackage{xcolor}

  \usepackage{amsthm}
\usepackage{amsmath, latexsym, amsfonts, amssymb, amscd,bbm,stmaryrd}
\usepackage[english]{babel}

\usepackage{mathrsfs}

\newtheorem{theorem} {Theorem} 
\newtheorem{lemma}[theorem]{Lemma}
\newtheorem{corollary}[theorem]{Corollary}

\newtheorem{hypothesis}{Assumption} 

\renewcommand{\tilde}{\widetilde}

\newcommand{\linf}[1]{\underset{#1\to\infty}{\underline{\lim}}}

 \newcommand{\lsup}[1]{\underset{#1\to\infty}{\overline{\lim}}}

\title{Large Deviations of Mean-Field Jump-Markov Processes on Structured Sparse Random Graphs}
\author{  J. MacLaurin }
 
\affiliation{organization = {New Jersey Institute of Technology},
addressline = {Martin Luther King Boulevard },
city = {Newark},
country = {USA},
ead = { james.maclaurin@njit.edu}
}

\begin{document}
 \begin{abstract} We prove a Large Deviation Principle for {\color{blue} jump-Markov } Processes on sparse large disordered network with disordered connectivity. The network is embedded in a geometric space, with the probability of a  connection a (scaled) function of the spatial positions of the nodes. This type of model has numerous applications, including neuroscience, epidemiology and social networks. We prove that the rate function (that indicates the asymptotic likelihood of state transitions) is the same as for a network with all-to-all connectivity. We apply our results to a stochastic $SIS$ epidemiological model on a disordered networks, and determine Euler-Lagrange equations that dictate the most likely transition path between different states of the network.
\end{abstract}
\maketitle



\renewcommand{\thefootnote}{\fnsymbol{footnote}}

\section{Introduction}

We study the dynamics of high-dimensional {\color{blue} jump-Markov } processes on networks {\color{blue} embedded in a geometric space}. {\color{blue} This class of model enjoys many applications, including in neuroscience \cite{DeMasi2014,Fournier2016,Locherbach2018,Chevallier2020,Avitabile2024,MacLaurin2026}, epidemics on structured populations \cite{Pellis2015,Riley2015,Allen2017}, sociological models \cite{Xing2024, lanchier2024stochastic} and models of opinion dynamics and belief propagation in structured networks \cite{lanchier2024stochastic}.} 

{\color{blue} To this end, we consider the dynamics of $N$ jump-Markov processes on an inhomogeneous network (i.e. a graph, with edges and nodes). The `agents' correspond to the nodes of the graph. The state of the $j^{th}$ agent is written as $\sigma^j(t)$, and it takes on values in a finite state-space $\Gamma$. As an example, for epidemiological applications \cite{Britton2019}, one typically takes $\Gamma = \lbrace S,I,R \rbrace$ (susceptible, infected, recovered) or just $\Gamma = \lbrace S, I \rbrace$. For neuroscience applications \cite{Gerstner2014,Chevallier2020}, one might take $\Gamma = \lbrace 0,1 \rbrace$ (i.e. spiking or non-spiking). {\color{blue} In many Poisson Point Process models, the spiking intensity is not Markovian, but depends on the entire history. To simplify the equations, we restrict ourselves to Markovian Processes. (We are already working with a complex model due to the spatial structure of the interactions, and the Large Deviation Principle requires additional work).

Broadly-speaking, our model can be thought of as a Markovian Hawkes Process. Hawkes Processes are continuous-time {\color{blue} Poisson Point Process } whose intensity function is itself stochastic \cite{Daley2003,Daley2008,Laub2021}. There has been much recent effort directed towards deriving deterministic `neural field' equations to describe the large {\color{blue} size limit } of  Hawkes Processes {\color{blue} embedded in a geometric space $\mathcal{E}$ } \cite{Delattre2016a,Chevallier2020,AgatheNerine2022,Coppini2024}, particulary in neuroscience.} {\color{blue} For this class of models, the spiking-intensity of any particular agent is a function of the average state of all the connected agents. As the system size $N$ asymptotes to infinity, the net input to any agent converges to be an integral over the geometric space $\mathcal{E}$ with respect to a connectivity kernel \cite{Chevallier2020,AgatheNerine2022,Coppini2024}.} A particular recent emphasis has been to study the formation of patterns and other coherent structures \cite{Goebel2021,Bramburger2023,Bramburger2024} in the limiting equations. 

Our paper builds on the works outlined in the previous paragraph by determining a Large Deviation Principle. The Large Deviation Principle determines an asymptotic estimate for {\color{blue} how the probability of a deviation from the limiting behavior depends on the size $N$ of the system}. It is useful for estimating large-scale transitions of the system induced by rare finite-size fluctuations of the system \cite{Bressloff2014a,Zakine2023,Grafke2024,Newby2024}. 

There exists a well-developed literature for the Large Deviations of PDEs perturbed by spatially-distributed Poisson Random Measures \cite{Budhiraja2013,Brzezniak2023,Budhiraja2019}. (Note that Hawkes Processes can be represented as a double-time integral with respect to a standard Poisson Random Measure on $([0,\infty))^2$ \cite{Lewis1978}.)  Our system differs from these papers in several respects: (i) the Poissonian noise can only occur at the locations of the nodes of the graph, (ii) there is a disordered network structure (iii) the interactions are not `local' like in a PDE, but are a mean-field function of the states of the neighboring nodes. Since the nodes are approximately uniformly distributed throughout the spatial domain, in the large size limit the Large Deviations rate function becomes an integral over $\mathcal{E} \times [0,\infty)$ of a Lagrangian function. 

On a related note, there has been much recent interest in the Large Deviations of Chemical Reaction Networks \cite{Patterson2019,Barbet2023}. These are high-dimensional jump Markovian Processes, just as in this paper, but without the network structure. Unlike this paper, in the large $N$ limiit, they do not obtain a rate function with spatial extension. Recent works include those of Dupuis, Ramanan and Wu \cite{DupuisRamananWu2016}, Pardoux and Samegni-Kepgnou \cite{Pardoux2017},  Agazzi, Eckmann and Dembo \cite{Agazzi2018},  and Patterson and Renger \cite{Patterson2019}. Patterson and Renger \cite{Patterson2019} and Agazzi, Patterson, Renger and \cite{Agazzi2022} prove the Large Deviations Principle in a general setting by also studying the convergence of the reaction fluxes (like in this paper). 

We also note that there exist works concerning the Large Deviations of neural fields perturbed by noise, see for example the work of Kuehn and Riedler \cite{Kuehn2014}. This paper differs from ours because the authors do not proceed from a microscopic network model, and they employ Brownian noise rather than Poissonian noise.
}
 
 The population is assumed to exist on a disordered static network. We make minimal assumptions on the network; in particular, we do not require that the edges are sampled from a probability distribution. Our main requirement is that the typical number of edges connected to each vertex asymptotes to $\infty$ as $N\to\infty$, and that the edge connectivity resembles that of a `graphon'. It is already known that these conditions ensure that the large $N$ limiting dynamics resembles the all-to-all connectivity case \cite{Lucon2020,AgatheNerine2022,Avitabile2024}. One way to generate the graphon structure is to sample the graph randomly from a probability distribution known as a W-random graph \cite{Borgs2005,Borgs2018}. Essentially this means that the connections are sampled independently, where the probability of a connection is a function of the locations of the afferent vertices. The formalism is flexibile enough to accommodate a wide range of models, including a power-law model (often considered a paradigm for populations with a clustered social structure), and populations with a geometric spatially-distributed (i.e. the probability of a connection correlates with the geometric distance between the people). 

Important cases covered by the $W$-random-graph  formalism include the following:
\begin{itemize}
\item \textit{Sparse Power Law Graphs}: these were originally defined in the seminal paper by Barabasi and Albert \cite{Barabasi1999}, and further developed by Bollobas \textit{et al} \cite{Bollobas2001}. In the original paper, Barabasi and Albert \cite{Barabasi1999} constructed this graph iteratively, by successively adding vertices, and then connecting them, with the probability of a connection being proportional to the existing degree of the node. It was shown by \cite{Borgs2018} that one obtains an asymptotically excellent approximation to a power law graph in the $W$-random graph formalism, as long as one chooses $\mathcal{E}=(0,1]$, $\lbrace x^j_N \rbrace_{j\in I_N}$ uniformly distributed over $\mathcal{E}$, and $\mathbb{P}\big( J^{jk} = 1 \big) = (1-\beta)^2 (x^j_N x^k_N)^{-\beta}$ for some parameters such that $0 < \beta < \gamma < 1$.
\item \textit{Inhomogeneous Erdos-Renyi Random Graphs}. There has recently been considerable interest in neuroscience and ecology for random graphs with distance-dependent connectivity \cite{Lucon2020}. In this example, one takes $\mathcal{E} = \mathbb{S}^1$ (motivation for this lies in the ring-structure of the visual cortex \cite{Bressloff2012}), $\mathcal{J}:\mathbb{S}^1 \times \mathbb{S}^1 \to [0,\infty)$ any smooth function
\item \textit{Small World Graphs} were first defined in the seminal paper of Watts and Strogatz \cite{Strogatz1998}. These Graphs are constructed by taking a ring of nearest-neighbor-connected vertices, and then randomly reassigning some edges.
\end{itemize}

\subsection{Structure of the Paper}
In Section 2, we outline our model, assumptions, and main results. In Section 3, we provide an extended example of a model of the spread of an epidemic through a structure population. In Section 4, we provide the proofs.

\subsection{Notation}

{\color{red} Let $\Xi \subseteq \Gamma \times \Gamma$ consist of all $(\alpha,\beta)$ such that (i) $\alpha\neq \beta$ and (ii) $f_{\beta}(x,\alpha,y) \neq 0$ for some $x \in \mathcal{E}$ and $y \in \mathbb{R}^{\Gamma}$.}

The $N$ particles are indexed by $I_N = \lbrace 1,2,\ldots,N-1,N \rbrace$.  If $\mathcal{E}$ is a Polish Space with metric $d_{\mathcal{E}}$, let $\mathcal{M}(\mathcal{E})$ denote the space of all Borel measures on $\mathcal{E}$. Let $\mathcal{P}(\mathcal{E}) \subseteq \mathcal{M}(\mathcal{E})$ denote the space of all measures with total mass of one. We endow both of these spaces with the topology of weak convergence: i.e. the topology generated by open sets of the form, for a continuous bounded function $g \in \mathcal{C}(\mathcal{E})$, $k\in \mathbb{R}$ and $\epsilon > 0$,
\[
\big\lbrace \mu \in \mathcal{M}(\mathcal{E}) : \big| \mathbb{E}^{\mu}[g] - k \big| < \epsilon \big\rbrace .
\]
The Wasserstein distance on $\mathcal{P}(\Gamma \times \mathcal{E})$ is defined to be
\begin{align}
d_W(\mu,\nu) = \inf\big\lbrace \mathbb{E}^{\zeta}\big[ \chi\lbrace \sigma \neq \tilde{\sigma} \rbrace + d_{\mathcal{E}}(x,\tilde{x}) \big] \big\rbrace,
\end{align}
and the infimum is over all couplings $\zeta$ of $\mu$ and $\nu$.  {\color{blue} We let $\chi$ be the indicator function. So for any event $\mathcal{A}$, if $\mathcal{A}$ holds then $\chi\lbrace \mathcal{A} \rbrace = 1$, and if it does not hold then $\chi\lbrace \mathcal{A} \rbrace = 0$.}

Write $\mathfrak{B}(\mathcal{E})$ to denote the Borel sigma algebra. Let $D\big( [0,T], \mathbb{R} \big)$ denote the Skorohod space of all cadlag functions.

Define the following metric $d_T$ on $ \mathcal{M}\big(  \mathcal{E} \times [0,T]\big)^{\Gamma \times \Gamma}$:
\begin{align}\label{eq: label bounded lipschitz metric}
d_T(\mu,\nu) = \sum_{\alpha,\beta \in \Gamma } \sup_{h} \big| \mathbb{E}^{\mu_{\alpha\mapsto \beta}}[h] - \mathbb{E}^{\nu_{\alpha\mapsto \beta}}[h] \big|,
\end{align}
where the supremum is taken over all functions $h$ that are (i) Lipschitz with Lipschitz constant less than or equal to $1$ and (ii) such that $|h| \leq 1$ \cite{Shiryaev2016}  . 

Write 
\[
\pi_T: \mathcal{M}\big(  \mathcal{E} \times [0, \infty) \big)^{\Gamma \times \Gamma} \to \mathcal{M}\big( \mathcal{E} \times [0,T]\big)^{\Xi}
\]
to be the projection of a measure onto its marginal upto time $T$. We endow $\mathcal{M}\big(  \mathcal{E} \times [0, \infty) \big)^{\Xi }$ with the cylinder topology generated by open sets of the form, for open $\mathcal{O} \subseteq \mathcal{M}\big(\mathcal{E} \times [0,T]\big)^{\Xi }$,
\begin{align}
\big\lbrace \mu \in \mathcal{M}\big( \mathcal{E} \times [0, \infty) \big)^{\Gamma \times \Gamma} : \pi_T(\mu) \in \mathcal{O} \big\rbrace .
\end{align}
We also naturally consider $d_T$ to be a semi-metric on $ \mathcal{M}\big(  \mathcal{E} \times [0,\infty) \big)^{\Gamma \times \Gamma}$, as follows
\begin{align}
d_T(\mu,\nu) := d_T\big( \pi_T \mu , \pi_T \nu \big).
\end{align}
Finally, define
\begin{align}
d(\mu,\nu)= \sum_{j=1}^{\infty} 2^{-j} d_j(\mu,\nu).
\end{align}

 \section{Model Outline and Main Result}
\subsection{Geometry of the Connectivity}
We make general assumptions on the geometry of the connectivity. These assumptions will hold in a variety of circumstances. The nodes of the graph are assigned positions on a compact smooth Riemannian Manifold $\mathcal{E}$. We let  $\mu_{Rie} \in \mathcal{P}(\mathcal{E})$ denote the (normalized) volume measure on $\mathcal{E}$. The position of particle $j \in I_N := \lbrace 1,2,\ldots, N \rbrace$ is denoted $x^j_N$ (a non-random constant that depends on $N$ too), and we write $x_N = (x^j_N)_{j\in I_N}$. Write the empirical measure of initial conditions to be
\begin{align}
\hat{\mu}^N(x_N) = N^{-1}\sum_{j \in I_N} \delta_{x^j_N} \in \mathcal{P}(\mathcal{E}).
\end{align}
\begin{hypothesis}
 It is assumed that $\hat{\mu}^N(x_N)$ converges weakly to some $\kappa \in  \mathcal{P}(\mathcal{E})$ as $N\to\infty$.
\end{hypothesis}
{\color{blue} By definition of a Riemannian manifold, there must exist a complete metric $d_{\mathcal{E}}$ that generates the topology on $\mathcal{E}$.} 

 The strength of connection from particle $j$ to particle $k$ is denoted $  J^{jk} \in \mathbb{R}$. The connectivity can be both excitatory and inhibitory; and symmetric or asymmetric. We will require that the connectivity converges to a `graphon'-type structure as $N\to\infty$ \cite{Lovasz2012,Avitabile2024}. {\color{blue} This will entail that there exists a continuous function $\mathcal{J} \in \mathcal{C}(\mathcal{E}\times\mathcal{E})$ such that the strength-of-connection $J^{jk}$ can be replaced by $\mathcal{J}(x^j_N , x^k_N)$ without a significant loss of accuracy, when $N$ is large.} {\color{blue} We do not necessarily require that the edges are sampled from a probability distribution, but in Lemma \ref{Lemma Sampled Edges} below we demonstrate that the assumptions could be satisfied by sampling the edges from a probability distribution.} These assumptions are broadly similar to those made by Lucon \cite{Lucon2020} in treating the large $N$ limiting dynamics for Brownian Motion-driven SDEs on disordered networks.
\begin{hypothesis}\label{hypothesis average convergence}
(i) We assume that there is a non-increasing sequence $(\phi_N)_{N\geq 1}$ and a constant $C_{\mathcal{J}} > 0$ such that for all $N\geq 1$,
\begin{align}
\sup_{j,k \in I_N}\big| J^{jk} \big| &\leq C_{\mathcal{J}} \\
\sup_{j\in I_N} \sum_{k\in I_N} \chi\lbrace J^{jk} \neq 0 \rbrace &\leq C_{\mathcal{J}} N\phi_N . 
\end{align}
(ii) It is assumed that 
\begin{align} \label{eq: big connectivity assumption}
\lim_{N\to\infty} N^{-1} \sum_{j \in I_N}  \eta^j_N = 0.
\end{align}
where
\begin{align}
\eta^j_N =    \sup_{\alpha \in \lbrace -1,0,1 \rbrace^N}\bigg| \sum_{k\in I_N} \big( \phi_N^{-1} J^{jk} - \mathcal{J}(x^j_N,x^k_N) \big) \alpha^k \bigg|.
\end{align}
(iii) Finally, we assume that $\mathcal{J}$ is uniformly Lipschitz in both arguments, i.e. for all $x,y,z \in \mathcal{E}$
\begin{align}
\big| \mathcal{J}(x,y) - \mathcal{J}(x,z)\big| &\leq C_{\mathcal{J}} d_{\mathcal{E}}(y,z) \\
\big| \mathcal{J}(y,x) - \mathcal{J}(z,x)\big| &\leq C_{\mathcal{J}} d_{\mathcal{E}}(y,z) .
\end{align}
\end{hypothesis}
We next note that Assumption 2 is guaranteed to be satisfied if the edges are sampled independently from a distribution whose probability varies continuously over $\mathcal{E}$.
\begin{lemma}\label{Lemma Sampled Edges}
{\color{blue} Suppose that $\lbrace J^{jk} \rbrace_{j,k \in I_N}$ assume values in $\lbrace -1, 0 , 1 \rbrace$ and that they have been sampled randomly from a distribution. Assume furthermore that they are either (i) mutually independent or (ii) independent, except that $J^{jk} = J^{kj}$. Suppose also that there are Lipschitz functions $p_+ , p_- : \mathcal{E} \times \mathcal{E} \to \mathbb{R}$ such that
 \begin{align}
 \mathbb{P}\big( J^{jk} = 1 \big) &= \phi_N p_+(x^j_N, x^k_N) \\
  \mathbb{P}\big( J^{jk} = -1 \big) &= \phi_N p_-(x^j_N, x^k_N) .
 \end{align}}
Suppose that the scaling is such that for any positive constant $c >0$,
 \begin{align}
 \lim_{N\to\infty} N \exp\big( -c N\phi_N \big) < \infty.
 \end{align}
 Then Assumption \ref{hypothesis average convergence} is satisfied, as long as we define
 \begin{equation}
  \mathcal{J}(\theta,\alpha) = p_+(\theta,\alpha) - p_-(\theta,\alpha).
\end{equation}
\end{lemma}
A proof is provided in \cite{Avitabile2024}. See also \cite{Lucon2020} for a similar model.




\subsection{Stochastic Transitions}
The transitions of the states are taken to be Poissonian and to assume the following mean-field form. For each $\alpha \in \Gamma$, it is assumed that there exists a function
\begin{align}
f_{\alpha}: \mathcal{E} \times \Gamma \times \mathbb{R}^{\Gamma}  \to [0,\infty),
\end{align}
such that for $\alpha \neq \sigma^j(t)$ and $h \ll 1$,
\begin{align}
\mathbb{P}\big( \sigma^j(t+h) = \alpha \; | \; \mathcal{F}_t \big) = h f_{\alpha}\big(x^j_N, \sigma^j(t) , w^j(t) \big) + O(h^2).
\end{align}
Here $w^j(t) = \big( w^j_{\beta}(t) \big)_{\beta \in \Gamma}$ and $w^j_{\beta}(t) \in [0,\infty)$ is such that
\begin{align}
w^j_{\beta}(t)  = N^{-1} \phi_N^{-1} \sum_{k=1}^N J^{jk}\chi\lbrace \sigma^k(t) = \beta \rbrace .
\end{align}
Define the empirical occupation measure at time $t$, $\hat{\nu}^N_{t} \in \mathcal{P}(\Gamma \times \mathcal{E} )$ to be such that for measurable $A \subseteq \mathcal{E}$,
\begin{align}
\hat{\nu}^N_{t}\big(\alpha \times A  \big) = N^{-1} \sum_{j\in I_N} \chi\lbrace x^j_N \in A , \sigma^j(t) = \alpha \rbrace .
\end{align}
The only assumption on the initial conditions $\lbrace \sigma^j(0) \rbrace_{j\in I_N}$ is that $\hat{\nu}^N_{0}$ converges weakly to a limit $\nu_0$, as noted in the following hypothesis.
\begin{hypothesis} \label{Lemma nu 0}
We assume that there is a measure  $\nu_0 \in \mathcal{P}(\Gamma \times \mathcal{E})$ such that 
\begin{align}
\lim_{N\to\infty} d_W\big(\hat{\nu}^N_0 , \nu_0 \big) = 0. 
\end{align}
$\nu_0$ is such that for any measurable subset $A \subseteq \mathcal{E}$ and $\zeta \in \Gamma$,
\begin{equation}
\nu_0(\zeta \times A ) = \int_{A} q_{0}(\zeta,x) \kappa(dx),
\end{equation}
where $x \mapsto q_{0}(\zeta,x)$ is continuous and for all $x \in \mathcal{E}$, $\sum_{\zeta\in\Gamma} q_0(\zeta,x) = 1$.
\end{hypothesis}
It must be emphasized that the initial conditions can be dependent on the specific choice of the connectivity. It is well-known that a Large Deviation Principle for Poisson Random Measures may not be possible if the intensity function hits zero \cite{Patterson2019,Agazzi2022}. Hence we make the following assumptions.
\begin{hypothesis}
Let $\Xi \subseteq \Gamma \times \Gamma$ consist of all $(\alpha,\beta)$ such that (i) $\alpha\neq \beta$ and (ii) $f_{\beta}(x,\alpha,y) \neq 0$ for some $x \in \mathcal{E}$ and $y \in \mathbb{R}^{\Gamma}$. It is assumed that $f_{\alpha}$ has strictly positive upper and lower bounds, i.e. there are constants $c_f,C_f > 0$ such that for all $(\beta,\alpha)\in\Xi$,
\begin{align}
0 < c_f \leq f_{\alpha}(\cdot,\beta,\cdot) \leq C_f.
\end{align}
Its also assumed that $f_{\alpha}$ is globally Lipschitz, i.e. for all $(\beta,\alpha) \in \Xi$ and all $w,\tilde{w} \in \mathbb{R}^{\Gamma}$,
\begin{align}
\big| f_{\alpha}(\theta ,\beta,w) - f_{\alpha}(\tilde{\theta} ,\beta,\tilde{w}) \big| \leq C_f \big\lbrace d_{\mathcal{E}}(\theta,\tilde{\theta}) + \| w - \tilde{w} \| \big\rbrace .
\end{align}
\end{hypothesis}

A key object used to study the large $N$ behavior of the system is the  empirical reaction flux \cite{Patterson2019}. For $(\alpha,\beta) \in \Xi$, the empirical reaction flux $\hat{\mu}^N_{\alpha \mapsto \beta } \in \mathcal{M}( \mathcal{E}   \times [0,\infty) )$ is defined to count the total number of $\alpha \mapsto \beta $ transitions over a specified time interval (and scaled by $N^{-1}$), i.e.  for a measurable subset $A \subset \mathcal{E}$ and $[a,b] \subset [0,\infty) $, 
\begin{align} \label{eq: transition empirical measure}
\hat{\mu}^N_{\alpha \mapsto \beta}\big( A  \times [a,b] \big) = N^{-1} \sum_{j\in I_N}\sum_{s  \in [a,b]} \chi\big\lbrace   x^j_N \in A ,  \sigma^j_{s^-} = \alpha , \sigma^j_{s} = \beta \big\rbrace .
\end{align}
and write $\hat{\mu}^N := \big( \hat{\mu}^N_{\alpha\mapsto\beta} \big)_{(\alpha,\beta)\in \Xi}$. When necessary we write $\hat{\mu}^N_T \in \mathcal{M}(\mathcal{E} \times [0,T])^{\Xi}$ to denote the measures upto finite times. 
The joint state space for the empirical reaction flux, and the empirical measure is denoted by 
\begin{align}
\mathcal{X} = \mathcal{M}\big(\mathcal{E} \times [0,\infty) \big)^{\Xi} \times \mathcal{D}\big( [0,\infty),  \mathcal{P}(\Gamma \times \mathcal{E}) \big).
\end{align}
\subsection{Main Results}
We first prove that the empirical reaction flux and the empirical occupation measure both concentrate in the large $N$ limit. (This is basically already known \cite{AgatheNerine2022,Coppini2024}).
\begin{theorem}\label{Theorem Large N limiting dynamics}
There exists unique $( \mu , \nu) \in \mathcal{X}$ to which the system converges as $N \to\infty$. Write the density of $\nu_t$ to be $q_{t}(\alpha,\theta)$, i.e.  for any $A \in \mathfrak{B}(\mathcal{E})$,
\begin{align}
\nu_t(\alpha \times A) :=& \int_A q_{t}(\alpha,\theta) d\kappa(\theta) .
\end{align}
The evolution of the density is such that for any $\alpha\in\Gamma$ and any $\theta \in \mathcal{E}$,
\begin{multline}
\frac{d q_{t} (\alpha,\theta)}{dt} =\sum_{\beta \in \Gamma: (\beta,\alpha) \in \Xi}  f_{\alpha}(\theta , \beta , w_t(\theta)) q_{t} (\beta,\theta) \\ - \sum_{\beta \in \Gamma: (\alpha,\beta) \in \Xi} f_{\beta}(\theta , \alpha , w_t(\theta))  q_{t} (\alpha,\theta)  , 
\end{multline}
where
\begin{align}
w_t(\theta) :=& \big( w_{t,\alpha}(\theta) \big)_{\alpha \in \Gamma} \; \;  \text{ and }\\ 
w_{t,\alpha}(\theta) =&  \int_{\mathcal{E}} \mathcal{J}(\theta, y ) q_{t}(\alpha,y) \kappa(dy).
\end{align}
Furthermore the limits of the empirical reaction fluxes are such that for $(\alpha,\beta) \in \Xi$,
\begin{align}
\frac{d \mu_{\alpha\mapsto \beta}}{dt}( A \times [s,t]) =& \int_A  f_{\beta}(\theta , \alpha , w_t(\theta)) q_{t}(  \alpha,\theta) \kappa(d\theta) \quad \quad t\geq s  \\
\mu_{\alpha\mapsto \beta}(A \times \lbrace s \rbrace) =& 0.
\end{align}
 \end{theorem}
We next state the Large Deviation Principle for the empirical reaction fluxes. We first define the rate function:
\begin{align}
\mathcal{G}: \mathcal{M}\big(  \mathcal{E} \times [0,\infty)\big)^{\Xi} \to \mathbb{R},
\end{align}
as follows. For $\mu \in \mathcal{M}\big(  \mathcal{E} \times [0,\infty) \big)^{\Xi}$, we stipulate that
\begin{align}
\mathcal{G}(\mu) = \infty
\end{align}
in the case that for some $(\alpha,\beta) \in \Xi$ with $\alpha\neq \beta$, $\mu_{\alpha \mapsto \beta}$ does not have a density (with respect to $\kappa \otimes \mu_{Leb}$). {\color{red} Here $\mu_{Leb}$ is Lebesgue measure on $[0,\infty)$}.  Otherwise, writing $p_{\alpha \mapsto \beta}(x,t)$ to be the density, i.e. the function such that
\begin{align}
\frac{d\mu_{\alpha \mapsto \beta}}{d \kappa \otimes d\mu_{Leb}}(x,t) = p_{\alpha \mapsto \beta}(x,t),
\end{align}
we define
\begin{align}
\mathcal{G}(\mu) =  \sum_{ (\alpha , \beta) \in \Xi} \int_{\mathcal{E}}    \int_0^\infty    \ell\bigg( \frac{p_{\alpha \mapsto \beta}(x,t) }{ \lambda_{\alpha\mapsto\beta}(x,t)} \bigg) \lambda_{\alpha\mapsto\beta}(x,t)   dt  \kappa(dx). \label{eq: G mu rate function}
\end{align}
Here, {\color{red} the convex function $\ell : \mathbb{R} \mapsto \mathbb{R}$ comes from the rate function for the homogeneous Poisson Process. Notice that the rate function can be written as a time-rescaling of the rate function for homogeneous Poisson Processes. The functions are}
\begin{align}
\ell(a) =& a\log a - a + 1 \\
\lambda_{\alpha\mapsto\beta}(x,t) =& f_{\beta}(x,\alpha, w (x,t) ) q_t(\alpha,x) \label{eq: lambda alpha beta x t} \\
q_t(\alpha,x) =& q_0(\alpha,x) + \int_0^t \bigg( \sum_{\beta: (\beta,\alpha) \in \Xi} p_{\beta \mapsto \alpha}(x,s) - \sum_{\beta: (\alpha,\beta) \in \Xi} p_{\alpha \mapsto \beta}(x,s) \bigg) ds \nonumber\\
w(x,t) =& \big(w_{\zeta}(x,t) \big)_{\zeta\in\Gamma} \\
w_{\zeta}(x,t) =&\int_{\mathcal{E}}   \mathcal{J}(x,y)q_t(\zeta,y) \kappa(dy).
\end{align}
For future reference, we briefly note the rate function restricted to finite time intervals. This is $\mathcal{G}_T: \mathcal{M}\big( \mathcal{E} \times [0,T] \big)^{\Xi} \mapsto \mathbb{R}$, which is defined to be such that
\begin{align}
\mathcal{G}_T(\mu) =  \sum_{ (\alpha , \beta) \in \Xi} \int_{\mathcal{E}}    \int_0^T    \ell\bigg( \frac{p_{\alpha \mapsto \beta}(x,t) }{\lambda_{\alpha\mapsto\beta}(x,t)} \bigg) \lambda_{\alpha\mapsto\beta}(x,t)   dt  \kappa(dx). \label{eq: G T mu definition}
\end{align}
The main theoretical result of this paper is the following theorem.
\begin{theorem} \label{Main Large Deviations Theorem}
Let $\mathcal{A}, \mathcal{O} \subseteq \mathcal{M}\big( \mathcal{E} \times [0,\infty) \big)^{\Xi}$ be (respectively) closed and open. Then
\begin{align}
\lsup{N} N^{-1}\log \mathbb{P}\big( \hat{\mu}^N \in \mathcal{A} \big) &\leq - \inf_{\mu \in \mathcal{A}} \mathcal{G}(\mu) \\
\linf{N} N^{-1}\log \mathbb{P}\big( \hat{\mu}^N \in \mathcal{O} \big) &\geq - \inf_{\mu \in \mathcal{O}} \mathcal{G}(\mu) .
\end{align}
Furthermore, $\mathcal{G}$ is lower semicontinuous and has compact level sets.
\end{theorem}

\subsection{Contracted Rate Function}

In this section, we determine the structure of the Large Deviation rate function for the empirical occupation measure $\hat{\nu}^N$ only.  One easily checks that the empirical occupation measures can be obtained by applying a continuous transformation to the empirical reaction fluxes. This is noted in the following Lemma.
\begin{lemma} \label{Lemma Psi Definition}
For $t > 0$ and $\nu_0 \in  \mathcal{P}(\Gamma \times \mathcal{E})$, define $\Psi_{\nu_0,t}:  \mathcal{M}\big( \mathcal{E} \times [0,\infty) \big)^{\Xi} \to \mathcal{P}(\Gamma \times \mathcal{E})$ to be such that for any $A \in \mathfrak{B}(\mathcal{E})$, and any $\alpha \in \Gamma$,
\begin{align*}
\nu_t(\alpha \times A) = \nu_0(\alpha \times A) +  \sum_{\beta : (\beta,\alpha) \in \Xi }   \mu_{\beta \mapsto \alpha}(A \times [0,t]) -  \sum_{\beta : (\alpha,\beta) \in \Xi } \mu_{\alpha\mapsto \beta}(A \times [0,t])  .
\end{align*}
For each $t > 0$ and $\nu_0 \in  \mathcal{P}(\Gamma \times \mathcal{E})$, $\Psi_{\nu_0,t}$ is continuous. Also $\nu_0 \mapsto \Psi_{\nu_0,t}(\mu)$ is continuous for any $\mu \in  \mathcal{M}\big( \mathcal{E} \times [0,\infty) \big)^{\Xi}$. Furthermore with unit probability, for all $t > 0$,
\begin{align}
\Psi_{\hat{\nu}^N_0 ,t}\big( \hat{\mu}^N\big) &= \hat{\nu}_t^N .
\end{align}
\end{lemma}
Write $\Psi_{\nu_0}(\mu) := (\Psi_{\nu_0,t}(\mu))_{t\geq 0} \in D\big( [0,\infty) , \mathcal{P}(\Gamma \times \mathcal{E}) \big)$. 
The proof of Lemma \ref{Lemma Psi Definition} follows almost immediately from the definitions and is neglected.

We can now define the contracted rate function (recalling that  $\nu_0$ is the limit of the empirical occupation measure at time $0$),

\begin{align}
\mathcal{H}&: \mathcal{D}\big( [0,\infty) , \mathcal{P}(\Gamma \times \mathcal{E}) \big) \to [0,\infty) \text{ where } \\
\mathcal{H}(\nu) &= \inf \big\lbrace \mathcal{G}(\mu) : \Psi_{\nu_0}( \mu) = \nu \text{ and }\mu \in \mathcal{M}(\mathcal{E} \times [0,\infty))^{\Xi}\big\rbrace ,
\end{align}
and we define $\mathcal{H}(\nu) = \infty$ if there does not exist any $\mu \in \mathcal{M}(\mathcal{E} \times [0,\infty))^{\Xi}$ such that $ \Psi_{\nu_0}( \mu) = \nu$. 

\begin{corollary} \label{Corollary Large Deviations}
Let $\mathcal{A}, \mathcal{O} \subseteq \mathcal{D}\big( [0,\infty),  \mathcal{P}( \Gamma \times \mathcal{E} ) \big)$ be (respectively) closed and open. Then
\begin{align}
\lsup{N} N^{-1}\log \mathbb{P}\big( \hat{\nu}^N \in \mathcal{A} \big) &\leq - \inf_{\nu \in \mathcal{A}} \mathcal{H}(\nu), \\
\linf{N} N^{-1}\log \mathbb{P}\big( \hat{\nu}^N \in \mathcal{O} \big) &\geq - \inf_{\nu \in \mathcal{O}} \mathcal{H}(\nu) .
\end{align}
Furthermore, $\mathcal{H}$ is lower semicontinuous and has compact level sets.
\end{corollary}
\begin{proof}
Since $\Psi$ is continuous, this follows from an application of the Contraction Principle \cite{Dembo1998} to Theorem \ref{Main Large Deviations Theorem}.
\end{proof}
We desire a more workable definition of $\mathcal{H}$. To this end, write
\begin{align}
\mathcal{U} \subseteq \mathcal{D}\big( [0,\infty) , \mathcal{P}(\Gamma \times \mathcal{E}) \big) 
\end{align}
to consist of all $(\nu_t)_{t\geq 0}$ such that (i) $t\mapsto \nu_t$ is continuous and (ii) there exist functions $\lbrace r_{t,\eta} \rbrace_{\eta\in\Gamma} \subset L^1(\mathcal{E} \times [0,\infty) )$ (the set of all functions that are integrable with respect to $\kappa \otimes \mu_{Leb}$ upto finite times) such that for all $A \in \mathfrak{B}(\mathcal{E})$,
\begin{align}
\nu_t(\eta \times A) =&\nu_0(\eta \times A)+ \int_0^t \int_{A} r_{s,\eta}(\theta) d\kappa(\theta) d\mu_{Leb}(s) \text{ and } \label{eq: nu t decomposition} \\
\sum_{\eta\in\Gamma} \nu_t(\eta \times A) =& \kappa(A).
\end{align}
For any $r \in L^1\big(\mathcal{E} \times [0,\infty) \big)^{\Gamma}$, written $r = \big(r_{t,\zeta}(x) \big)_{x\in \mathcal{E}, t\geq 0, \zeta \in \Gamma}$, define the set
\begin{multline}
\mathcal{Q}_t(r) = \bigg\lbrace \big( q_{\alpha \mapsto\beta} \big)_ {(\alpha,\beta)\in \Xi} \in  L^1( \mathcal{E})^{\Xi} \; : \\ r_{t,\zeta} (x)= \sum_{\alpha\in\Gamma : (\alpha,\zeta) \in \Xi} q_{\alpha\mapsto\zeta}(x) - \sum_{\alpha\in\Gamma : (\zeta,\alpha) \in \Xi}  q_{\zeta\mapsto\alpha}(x)  \bigg\rbrace,
\end{multline}
and define the function $L: L^1(\mathcal{E})^{\Gamma} \times L^1(\mathcal{E})^{\Gamma} \mapsto \mathbb{R}$ to be
\begin{multline}
L_t( r,w) = \inf\bigg\lbrace \sum_{ (\alpha,\beta) \in \Xi} \int_{\mathcal{E}} \ell\big(q_{\alpha\mapsto\beta}(x) / \lambda_{\beta}(x,\alpha, w(x)) \big)  \lambda_{\beta}(x,\alpha, w(x)) \kappa(dx) \\ \text{ where } q \in \mathcal{Q}_t(r)  \bigg\rbrace .
\end{multline}
\begin{lemma}
If $ \nu  \notin \mathcal{U}$, then
\begin{align}
\mathcal{H}(\nu) = \infty.
\end{align}
Otherwise, if $\nu \in \mathcal{U}$, and letting $r_t$ satisfy \eqref{eq: nu t decomposition}, it holds that
\begin{align}
\mathcal{H}(\nu) =& \int_0^{\infty} L_t(r_t,w_t) dt \text{ where }\\
w_t(x) =& \big(w_{\zeta}(x,t) \big)_{\zeta\in\Gamma} \text{ and } \\
w_{\zeta}(x,t) =& \int_{\mathcal{E}}  \mathcal{J}(x,y) \chi\lbrace \alpha=\zeta \rbrace d\nu_t(\alpha,y) .
\end{align}
Furthermore, $\mathcal{H}$ is strictly convex in its first argument.
\end{lemma}




\section{Proofs}

There are two main steps to our proof of Theorem \ref{Main Large Deviations Theorem}. The first step is to show that the system can be approximated very well by a system with averaged interactions. The next step is to prove the Large Deviation Principle for the system with averaged interactions (this is Theorem \ref{Main Large Deviations Theorem Averaged}). The main result of this paper (Theorem \ref{Main Large Deviations Theorem}) will follow from these results thanks to \cite[Theorem 4.2.13]{Dembo1998}.

\subsection{Proof Outline}

We start by defining an approximate process with `averaged' interactions. The Large Deviations of this system will be indistinguishable from the original system in the large $N$ limit.
\subsection{System with Averaged Interactions} \label{Section Averaged System}



Let $\lbrace \bar{\sigma}^{j}(t) \rbrace_{j\in I_N}$ be a system of jump-Markov Processes such that, for $\alpha \neq \bar{\sigma}^{j}(t) $, and $h\ll 1$,
\begin{align}
 \mathbb{P}\big( \bar{\sigma}^j(t+h) = \alpha \; | \; \mathcal{F}_t \big) = h f_{\alpha}\big(x^j_N, \bar{\sigma}^j(t) ,  \bar{w}^{j}(t)  \big)+ O(h^2),
\end{align}
where $\bar{w}^j(t) = \big( \bar{w}^j_{\beta}(t) \big)_{\beta \in \Gamma}$ and
\begin{align}
\bar{w}_{\beta}^j(t) = N^{-1}  \sum_{k=1}^N \mathcal{J}(x^j_N , x^k_N) \chi\lbrace \bar{\sigma}^k(t) = \beta \rbrace .
\end{align}
We take the initial conditions to be the same for the two systems, i.e. $\bar{\sigma}^j(0) = \sigma^j(0)$. {\color{blue} We use the convenient time-rescaled representation of the Hawkes Process \cite{Ethier1986, Anderson2015}. Let $\big\lbrace Y^j_{\alpha\beta}(t) \big\rbrace_{(\alpha,\beta)\in \Xi}$ be independent unit-intensity counting processes.} Then define $\lbrace Z^j_{\alpha\beta}(t) \rbrace_{(\alpha,\beta)\in \Xi }$ to `count' the number of $\alpha \mapsto \beta$ transitions in the coupled system, i.e. be such that
\begin{align}
Z^j_{\alpha\beta}(t) =& Y^j_{\alpha\beta}\bigg(\int_0^t f_{\alpha}\big(x^j_N, \bar{\sigma}^j(s) ,  \bar{w}^{j}(s)  \big) \chi\lbrace \bar{\sigma}^j_s = \alpha \rbrace ds \bigg),  \label{eq: Z j alpha beta definition}
\end{align}
and for any $\alpha \in \Gamma$,
\begin{align}
\bar{\sigma}^j(t) =& \alpha \text{ if and only if }\\
 \chi\lbrace \bar{\sigma}^j(0) = \alpha \rbrace + \sum_{(\beta,\alpha) \in \Xi}  Z^j_{\beta\alpha}(t) - \sum_{(\alpha,\beta) \in \Xi}  Z^j_{\alpha\beta}(t)  =& 1. \label{eq: count the transitions}
\end{align}
Since (with unit probability) $Y^j_{\alpha\beta}(t)$ only makes a finite number of jumps over a bounded time interval, one easily checks (through iteratively solving the system from jump to jump) that there exists a unique $\lbrace \bar{\sigma}^j(t) \rbrace_{j\in I_N \fatsemi t \geq 0}$ satisfying \eqref{eq: Z j alpha beta definition} - \eqref{eq: count the transitions}.

We define the empirical reaction flux $\bar{\mu}^N_{\alpha\mapsto\beta} \in \mathcal{M}\big( \mathcal{E} \times [0,\infty)\big)$ to be such that for any $A \in \mathfrak{B}(\mathcal{E})$ and an interval $[a,b] \subset [0,\infty)$,
\begin{align}
\bar{\mu}^N_{\alpha\mapsto\beta}\big( A \times [a,b] \big) = N^{-1}\sum_{j\in I_N} \sum_{t \in [a,b]} \chi\big\lbrace   x^j_N \in A , Z^j_{\alpha\beta}(t^-) \neq  Z^j_{\alpha\beta}(t) \big\rbrace .
\end{align}
We write $\bar{\mu}^N = \big( \bar{\mu}^N_{\alpha\mapsto\beta} \big)_{\alpha,\beta\in\Gamma} \in  \mathcal{M}\big( \mathcal{E} \times [0,\infty)\big)^{\Gamma \times \Gamma}$. Our first goal is to prove a Large Deviation Principle for this averaged system. The rate function $\mathcal{G}_T$ is defined earlier in \eqref{eq: G mu rate function}.
\begin{theorem} \label{Main Large Deviations Theorem Averaged}
Let $\mathcal{A}, \mathcal{O} \subseteq \mathcal{M}\big(  \mathcal{E} \times [0,T]\big)^{ \Xi}$ be (respectively) closed and open. Then
\begin{align}
\lsup{N} N^{-1}\log \mathbb{P}\big( \bar{\mu}^N_T \in \mathcal{A} \big) &\leq - \inf_{\mu \in \mathcal{A}} \mathcal{G}_T(\mu) \\
\linf{N} N^{-1}\log \mathbb{P}\big( \bar{\mu}^N_T \in \mathcal{O} \big) &\geq - \inf_{\mu \in \mathcal{O}} \mathcal{G}_T(\mu) . \label{LDP Lower Bound bar mu N}
\end{align}
Furthermore, $\mathcal{G}_T$ is lower semicontinuous and has compact level sets.
\end{theorem}
Theorem \ref{Main Large Deviations Theorem Averaged} will be proved further below, in Section \ref{Section Prove LDP Averaged System}.

Now in order that our Main Theorem holds, in fact it suffices that we prove it for arbitrarily long times. This is stated in the following Lemma.

\begin{lemma} \label{Main Large Deviations Theorem finite times}
Let $\mathcal{A}, \mathcal{O} \subseteq \mathcal{M}\big( \mathcal{E} \times [0,T] \big)^{\Xi}$ be (respectively) closed and open. Then
\begin{align}
\lsup{N} N^{-1}\log \mathbb{P}\big( \hat{\mu}^N_T \in \mathcal{A} \big) &\leq - \inf_{\mu \in \mathcal{A}} \mathcal{G}_T(\mu) \\
\linf{N} N^{-1}\log \mathbb{P}\big( \hat{\mu}^N_T \in \mathcal{O} \big) &\geq - \inf_{\mu \in \mathcal{O}} \mathcal{G}_T(\mu) .
\end{align}
Furthermore, $\mathcal{G}$ is lower semicontinuous and has compact level sets.
\end{lemma}

We next state the proof of our main result, Theorem \ref{Main Large Deviations Theorem}.
\begin{proof}
Recall the continuous projection of a measure onto its restriction upto time $T$, i.e.
\begin{equation}
\pi_T:  \mathcal{M}\big(  \mathcal{E} \times [0, \infty )  \big)^{\Xi } \mapsto  \mathcal{M}\big(  \mathcal{E} \times [0,T] \big)^{\Xi }.
 \end{equation}
 Furthermore, the topology on $ \mathcal{M}\big(  \mathcal{E} \times[0, \infty )  \big)^{\Xi }$ is generated by open sets of the form
 \begin{align}
 \bigg\lbrace \mu \in  \mathcal{M}\big(  \mathcal{E} \times[0, \infty )  \big)^{\Xi } \; : \; \pi_T(\mu) \in \mathcal{O}_T \bigg\rbrace
 \end{align}
 for some open $\mathcal{O}_T \subseteq \mathcal{M}\big(  \mathcal{E} \times [0,T] \big)^{\Xi }$.
 
 The Dawson-Gartner Projective Limits Theorem \cite{Dembo1998} means that the result of Lemma \ref{Main Large Deviations Theorem finite times} extends to a Large Deviations Principle over 
 $ \mathcal{M}\big(  \mathcal{E} \times[0, \infty )  \big)^{\Xi }$, with rate function
 {\color{red}
 \begin{align}
\tilde{\mathcal{G}}_T(\mu) := \lim_{T\to\infty} \mathcal{G}_T\big( \pi_T\mu \big).
 \end{align} }
 However it is immediate from the definitions that
 \begin{equation}
  \mathcal{G}(\mu) = \tilde{\mathcal{G}}_T(\mu).
  \end{equation}
\end{proof}

We next use Girsanov's Theorem to compare the Large Deviations in our main result (Theorem \ref{Main Large Deviations Theorem})  to the Large Deviation Principle in Theorem \ref{Main Large Deviations Theorem Averaged}. Let $\bar{P}^N_T \in \mathcal{P}\big( \mathcal{D}([0,T], \Gamma)^N \big)$ be the probability law of $\big( \bar{\sigma}^j_t \big)_{j\in I_N \fatsemi t\leq T}$. Let 
$P^N_{J,T} \in \mathcal{P}\big( \mathcal{D}([0,T], \Gamma)^N \big)$ be the probability law of the original system $\big( \sigma^j_t \big)_{j\in I_N \fatsemi t\leq T}$.
Thanks to Girsanov's Theorem \cite{Bremaud1981},
\begin{equation}
\frac{dP^N_{J,T}}{d\bar{P}^N_T}(\sigma) = \exp\big( N \Phi_T(\sigma) \big) \label{eq: Girsanov Expression Gamma T}
\end{equation}
where
\begin{multline}
\Phi_T(\sigma) = N^{-1} \sum_{j\in I_N} \int_0^T  \sum_{\beta: (\sigma^j(s) , \beta ) \in \Xi }\big(   f_{\beta}(x^j_N , \sigma^j(s) , \bar{w}^j_s)  - f_{\beta}(x^j_N , \sigma^j(s) , w^j_s)  \big) ds \\ + N^{-1} \sum_{j\in I_N} \sum_{s \leq T \; : \; \sigma^j(s^-) \neq \sigma^j(s) } \bigg\lbrace \log  \big(  f_{\sigma^j(s)}(x^j_N , \sigma^j(s^-) , w^j_{s})  \big) \\-  \log  \big(  f_{\sigma^j(s)}(x^j_N , \sigma^j(s^-) , \bar{w}^j_{s})  \big) 
 \bigg\rbrace .
\end{multline}
 In the following lemma we prove that the Girsanov Exponent is with very high probability uniformly upperbounded.
\begin{lemma}\label{Lemma Girsanov Comparison}
For any $\epsilon,T > 0$,
\begin{align}
\lsup{N} N^{-1} \log \mathbb{P}\big( \big| \Phi_T(\sigma) \big| \geq \epsilon \big) = - \infty.
\end{align}
\end{lemma}
We can now state the proof of Lemma \ref{Main Large Deviations Theorem finite times}.
\begin{proof}
Let 
\[
\mathcal{V}_{N,\epsilon} = \big\lbrace \big| \Phi_T(\sigma) \big| \leq \epsilon  \big\rbrace .
\]
Starting with the upper bound, let $\mathcal{A} \subset \mathcal{M}\big(\mathcal{E} \times [0,\infty) \big)^{\Xi}$ be closed. Then for any $\epsilon > 0$,
\begin{multline}
\lsup{N} N^{-1}\log \mathbb{P}\big( \hat{\mu}^N \in \mathcal{A} \big) \leq \\
 \max\bigg\lbrace \lsup{N} N^{-1}\log \mathbb{P}\big( \hat{\mu}^N \in \mathcal{A} , \mathcal{V}_{N,\epsilon} \big) ,  \lsup{N} N^{-1}\log \mathbb{P}\big( \mathcal{V}_{N,\epsilon}^c \big) \bigg\rbrace
\end{multline}
The second term on the RHS is $-\infty$, thanks to Lemma \ref{Lemma Girsanov Comparison}. It thus suffices that we demonstrate that
\begin{align} \label{eq: to show epsilon zero G mu}
\lim_{\epsilon \to 0^+}  \lsup{N} N^{-1}\log \mathbb{P}\big( \hat{\mu}^N \in \mathcal{A} , \mathcal{V}_{N,\epsilon} \big)  \leq - \inf_{\mu \in \mathcal{A}} \mathcal{G}(\mu).
 \end{align}
 Now thanks to the Girsanov Expression in \eqref{eq: Girsanov Expression Gamma T}
 \begin{align}
  \lsup{N} N^{-1}\log \mathbb{P}\big( \hat{\mu}^N \in \mathcal{A} , \mathcal{V}_{N,\epsilon} \big)  \leq & \epsilon +   \lsup{N} N^{-1}\log \mathbb{P}\big( \bar{\mu}^N \in \mathcal{A} , \mathcal{V}_{N,\epsilon} \big) \nonumber \\
  \leq &  \epsilon +   \lsup{N} N^{-1}\log \mathbb{P}\big( \bar{\mu}^N \in \mathcal{A}   \big) \\
  \leq & \epsilon - \inf_{\mu \in \mathcal{A}} \mathcal{G}(\mu),
  \end{align}
  thanks to Theorem \ref{Main Large Deviations Theorem Averaged}. Taking $\epsilon$ to $0$, we obtain \eqref{eq: to show epsilon zero G mu}.
  
  Turning to the lower bound, let $\mathcal{O} \subset \mathcal{M}\big(\mathcal{E} \times [0,\infty) \big)^{\Gamma \times \Gamma}$ be open, we find that for any $\epsilon > 0$,
  \begin{align}
  \linf{N} N^{-1}\log \mathbb{P}\big( \hat{\mu}^N \in \mathcal{O}  \big) \geq &  \linf{N} N^{-1}\log \mathbb{P}\big( \hat{\mu}^N \in \mathcal{O} , \mathcal{V}_{N,\epsilon} \big) \nonumber \\
  \geq & \linf{N} N^{-1}\log \mathbb{P}\big( \bar{\mu}^N \in \mathcal{O} , \mathcal{V}_{N,\epsilon} \big) - \epsilon,
  \end{align}
  thanks to \eqref{eq: Girsanov Expression Gamma T}. Now
  \begin{align}
  \mathbb{P}\big( \bar{\mu}^N \in \mathcal{O} , \mathcal{V}_{N,\epsilon} \big)  = \mathbb{P}\big( \bar{\mu}^N \in \mathcal{O} \big) - \mathbb{P}\big( \bar{\mu}^N \in \mathcal{O} , \mathcal{V}^c_{N,\epsilon} \big) 
  \end{align}
  and since $\linf{N} N^{-1}\log \mathbb{P}\big( \bar{\mu}^N \in \mathcal{O} , \mathcal{V}^c_{N,\epsilon} \big) = -\infty$, it must hold that
  \begin{align}
  \linf{N} N^{-1}\log \mathbb{P}\big( \bar{\mu}^N \in \mathcal{O} , \mathcal{V}_{N,\epsilon} \big) \geq & \linf{N} N^{-1} \log  \mathbb{P}\big( \bar{\mu}^N \in \mathcal{O} \big) \nonumber \\
\geq & - \inf_{\mu \in \mathcal{O}}\mathcal{G}(\mu).
  \end{align}
  Taking $\epsilon \to 0^+$, we obtain that 
    \begin{align}
  \linf{N} N^{-1}\log \mathbb{P}\big( \hat{\mu}^N \in \mathcal{O}  \big) \geq  - \inf_{\mu \in \mathcal{O}}\mathcal{G}(\mu),
  \end{align}
  as required.
\end{proof}

We next prove Lemma \ref{Lemma Girsanov Comparison}.
\begin{proof}
It suffices to demonstrate the following two inequalities
\begin{align}
\lsup{N}& N^{-1} \log \mathbb{P}\bigg(  N^{-1} \sum_{j\in I_N}  \int_0^T\sum_{\beta\in \Gamma : (\sigma^j_s,\beta) \in \Xi} \big(   f_{\beta}(x^j_N , \sigma^j(s) , \bar{w}^j_s)  \nonumber\\ &- f_{\beta}(x^j_N , \sigma^j(s) , w^j_s)  \big) ds \geq \epsilon / 3 \bigg) = -\infty \label{eq: to compare lemma 1} \\
\lsup{N}& N^{-1} \log \mathbb{P}\bigg(  N^{-1} \sum_{j\in I_N} \sum_{s \leq T \; : \; \sigma^j(s^-) \neq \sigma^j(s) }  \bigg\lbrace \log\big(  f_{\sigma^j(s)}(x^j_N , \sigma^j(s^-) , w^j_{s^-})  \big)\nonumber \\ & -  \log\big(  f_{\sigma^j(s)}(x^j_N , \sigma^j(s^-) , \bar{w}^j_{s^-})  \big) \bigg\rbrace \geq \epsilon /3 \bigg) = -\infty  \label{eq: to compare lemma 2} 
\end{align}

For each $\beta \in \Gamma$, it follows from the fact that $f_{\beta}$ is Lipschitz  that there is a constant $C > 0$ such that
\begin{align*}
N^{-1} \sum_{j\in I_N} \big| f_{\beta}(x^j_N , \sigma^j(s) , \bar{w}^j_s)  - f_{\beta}(x^j_N , \sigma^j(s) , w^j_s) \big| 
\leq N^{-1}C\sum_{j\in I_N} \| \bar{w}^j_s - w^j_s \| .
\end{align*}
Furthermore Assumption \ref{hypothesis average convergence} implies that there must exist a non-random sequence $(\delta_N)_{N\geq 1}$ that decreases to $0$ and such that 
\begin{align}
N^{-1}\sum_{j\in I_N} \| \bar{w}^j_s - w^j_s \|  \leq \delta_N.
\end{align}
Once $N$ is large enough that $CT\delta_N < \epsilon/ 3$, \eqref{eq: to compare lemma 1} must hold.

{\color{blue} Turning to \eqref{eq: to compare lemma 2}, we will employ Assumption 4, i.e.  $f_{\beta}$ is (i) Lipschitz and (ii) uniformly lower-bounded by a positive constant and (iii) uniformly upperbounded by a constant $f_{max}$}. There thus exists a constant $C > 0$ such that (for the constant $\eta^j_N$ defined in Assumption \ref{hypothesis average convergence})
\begin{align}
\bigg| N^{-1} \sum_{j\in I_N} \sum_{s \leq T \; : \; \sigma^j(s^-) \neq \sigma^j(s) } & \bigg\lbrace \log\big(  f_{\sigma^j(s) }(x^j_N , \sigma^j(s^-) , w^j_s)  \big) \nonumber \\ &-  \log\big(  f_{\sigma^j(s) }(x^j_N , \sigma^j(s^-) , \bar{w}^j_s)  \big) \bigg\rbrace \bigg|  \nonumber \\
\leq &C N^{-1} \sum_{j\in I_N} \sum_{s \leq T \; : \; \sigma^j(s^-) \neq \sigma^j(s) } \eta^j_N \nonumber  \\
\leq &C N^{-1} \sum_{j\in I_N}\sum_{(\alpha,\beta) \in \Xi } Z^j_{\alpha\beta}(T) {\color{blue} \eta^j_N} \nonumber \\
\leq  &C N^{-1} \sum_{j\in I_N}\sum_{(\alpha,\beta) \in \Xi } Y^j_{\alpha\beta}( f_{max} T) \eta^j_N .
\end{align}
Thanks to Chernoff's Inequality, for a constant $a > 0$,{\color{blue}
\begin{multline}
\mathbb{P}\bigg( C N^{-1} \sum_{j\in I_N}\sum_{(\alpha,\beta) \in \Xi }  Y^j_{\alpha\beta}( f_{max} T) \eta^j_N \geq \epsilon \bigg) \nonumber \\ 
\leq \mathbb{E}\bigg[ \exp\bigg( a   \sum_{j\in I_N}\sum_{(\alpha,\beta) \in \Xi } Y^j_{(\alpha,\beta)}( f_{max} T) \eta^j_N - a C^{-1} N \epsilon \bigg) \bigg] \nonumber \\
\leq \exp\bigg(   |\Xi|   f_{max}T \sum_{j\in I_N} \big( \exp(a \eta^j_N )-1  \big)  - aC^{-1} N \epsilon \bigg) ,
\end{multline}
using the standard expression for the moment generating function of the Poisson Distribution.}

We next claim that for arbitrarily large $a$
\begin{equation}
\lim_{N\to\infty}\bigg\lbrace N^{-1} |\Xi|   f_{max}T \sum_{j\in I_N} \big( \exp(a \eta^j_N )-1  \big) \bigg\rbrace  = 0. \label{eq: claim for large a}
\end{equation}
Now Assumption \ref{hypothesis average convergence} implies that there exists a non-random constant $C_{\mathcal{J}}$ such that $\eta^j_N \leq \mathcal{C}_{\mathcal{J}}$ with unit probability. We thus obtain that, for $\delta$ small enough that for all $b\in [0,\delta]$, {\color{blue} $\exp(ab  )-1 \leq 2ab$}, 
\begin{align*}
N^{-1} \sum_{j\in I_N} \big( \exp(a \eta^j_N )-1  \big)   \leq & N^{-1}\sum_{j\in I_N} 2a\eta^j_N \chi \big\lbrace \eta^j_N \leq \delta  \big\rbrace \nonumber \\ &+N^{-1}\sum_{j\in I_N} \exp\big( a C_{\mathcal{J}} \big)  \chi \big\lbrace \eta^j_N \geq \delta \big\rbrace  \\
\leq &N^{-1}\sum_{j\in I_N} 2a\eta^j_N + N^{-1}\sum_{j\in I_N} \exp\big( a C_{\mathcal{J}} \big)  \chi \big\lbrace \eta^j_N \geq  \delta  \big\rbrace .
\end{align*}
Assumption \ref{hypothesis average convergence} implies that 
\begin{align}
\lim_{N\to\infty} N^{-1}\sum_{j\in I_N}  \chi \big\lbrace \eta^j_N \geq \delta \big\rbrace &= 0 \\
\lim_{N\to\infty} N^{-1}\sum_{j\in I_N} 2a\eta^j_N  &= 0.
\end{align}
We have thus established \eqref{eq: claim for large a}. 

We therefore find that for arbitrarily large $a$,
\begin{align}
\lsup{N} N^{-1} \log \mathbb{P}\bigg( C N^{-1} \sum_{j\in I_N}\sum_{(\alpha,\beta) \in \Xi} Y^j_{\alpha,\beta}( f_{max} T) \eta^j_N \geq \epsilon \bigg)  \leq - a C^{-1}.
\end{align}
Taking $a \to \infty$, we obtain \eqref{eq: to compare lemma 2}. We have thus established \eqref{eq: to compare lemma 1} and \eqref{eq: to compare lemma 2}.
\end{proof}
For a reaction $\alpha \mapsto \beta$, define the empirical flux measure $\bar{\mu}^N_{\alpha \mapsto \beta} \in \mathcal{M}(\mathcal{E} \times [0,\infty))$ for the averaged system to be such that, for measurable $A \subseteq \mathcal{E}$ and a time interval $[a,b]$,
\begin{align} \label{eq: transition empirical measure averaged}
\bar{\mu}^N_{\alpha \mapsto \beta }\big(  A \times [a,b] \big) = N^{-1} \sum_{j\in I_N}\sum_{s  \in [a,b]} \chi\big\lbrace x^j_N \in A ,  \bar{\sigma}^j_{s^-} = \alpha , \bar{\sigma}^j_{s} = \beta \big\rbrace .
\end{align} 

Writing $\delta_m = \big(\log m \big)^{-1}$, define the set
\begin{multline} \label{eq: K m set}
\mathcal{K}_m = \bigg\lbrace \mu \in \mathcal{M}\big( \mathcal{E} \times [0,\infty) \big)^{\Xi} \; :  \\  \text{ For all }0\leq b \leq m^2, \;  \mu_{\alpha \mapsto \beta} \big( \mathcal{E} \times [ b/m , (b+1) / m] \big) \leq \delta_m \bigg\rbrace .
\end{multline}
\begin{lemma} \label{Lemma K m bound}
For any $m\in \mathbb{Z}^+$, there exists $N_m$ such that for all $N\geq N_m$,
\begin{align}
N^{-1}\log \mathbb{P}\big(\bar{\mu}^N \notin \mathcal{K}_m  \big) \leq - \frac{1}{2} \sqrt{\log m}
\end{align}
\end{lemma}
\begin{proof}
Write $t^{(m)}_a = a  / m$. Using a union of events bound,
\begin{align}
 \mathbb{P}\big(\bar{\mu}^N \notin \mathcal{K}_m  \big) \leq \sum_{a=0}^{m^2} \mathbb{P}\big( N^{-1}\sup_{\alpha,\beta\in\Gamma} \sum_{j\in I_N} \big( Z^j_{(\alpha,\beta)}\big( t^{(m)}_{a+1} \big) -  Z^j_{(\alpha,\beta)}\big( t^{(m)}_{a} \big) \big) \geq \delta_m \big)
\end{align}
For a positive integer $a$, and $c > 0$, thanks to Chernoff's Inequality,
\begin{align}
\mathbb{P}\big( N^{-1}\sup_{(\alpha,\beta)\in \Xi }& \sum_{j\in I_N} \big( Z^j_{\alpha\beta}\big( t^{(m)}_{a+1} \big) -  Z^j_{\alpha\beta}\big( t^{(m)}_{a} \big) \big) \geq \delta_m \big) \nonumber \\
&\leq \sum_{(\alpha,\beta)\in \Xi } \mathbb{P}\big( N^{-1}  \sum_{j\in I_N} \big( Z^j_{\alpha\beta}\big( t^{(m)}_{a+1} \big) -  Z^j_{\alpha\beta}\big( t^{(m)}_{a} \big) \big) \geq \delta_m \big) \nonumber \\
&\leq \sum_{(\alpha,\beta)\in \Xi } \mathbb{E}\bigg[ \exp\bigg( c \sum_{j\in I_N} \big( Z^j_{\alpha\beta}\big( t^{(m)}_{a+1} \big) -  Z^j_{\alpha\beta}\big( t^{(m)}_{a} \big) \big) -c N \delta_m \bigg) \bigg] \nonumber \\
&\leq \sum_{(\alpha,\beta)\in \Xi } \mathbb{E}\bigg[ \exp\bigg( c \sum_{j\in I_N} \big( Y^j_{\alpha\beta}\big( t^{(m)}_{a} + f_{max} / m \big) -  Y^j_{\alpha\beta}\big( t^{(m)}_{a} \big) \big)-c N \delta_m \bigg) \bigg] \nonumber \\
&= \exp\big( -c N \delta_m \big)  \sum_{ (\alpha,\beta)\in \Xi }  \bigg( 1 +f_{max}m^{-1} \big( \exp(c) - 1 \big) \bigg)^N.
\end{align}
Thanks to the inequality $1+x \leq \exp(x)$,
\begin{multline}
 \exp\big( -c N \delta_m \big)  \sum_{(\alpha,\beta)\in \Xi }  \bigg( 1 + f_{max}m^{-1} \big( \exp(c) - 1 \big) \bigg)^N \leq \\ 
 | \Gamma |^2 \exp\bigg( -c N \delta_m + N  f_{max}m^{-1} \big( \exp(c) - 1 \big) \bigg). 
\end{multline}
We choose $\delta_m = (\log m)^{-1/2}$ and $c = \log m$, and we obtain that {\color{blue}
\begin{align}
N^{-1}\log \mathbb{P}\big( \bar{\mu}^N \notin \mathcal{K}_m  \big) \leq - \sqrt{\log m} / 2,
\end{align} 
as long as $m$ is large enough that $ \sqrt{ \log m } > \sqrt{\log m} / 2 + f_{max}$. }
\end{proof}

\subsection{Proof of the Large Deviations of the Averaged System} \label{Section Prove LDP Averaged System}

In this section we prove Theorem \ref{Main Large Deviations Theorem Averaged}. Our proof proceeds by first proving the Large Deviation Principle for a spatially-discretized system, before showing that the spatially-discretized system is an excellent approximation to the spatially-continuous system. 

Our main aim in this section will be to prove Theorem \ref{Main Large Deviations Theorem Averaged} , which we now restate.
\begin{theorem} \label{Lemma finite time intervals LDP}
Let $\mathcal{A}, \mathcal{O} \subseteq \mathcal{M}\big(  \mathcal{E} \times [0,T] \big)^{\Xi }$ be (respectively) closed and open. Then
\begin{align}
\lsup{N} N^{-1}\log \mathbb{P}\big( \bar{\mu}^N \in \mathcal{A} \big) &\leq - \inf_{\mu \in \mathcal{A}} \mathcal{G}_T(\mu)\label{LDP Upper Bound bar mu N finite T} \\
\linf{N} N^{-1}\log \mathbb{P}\big( \bar{\mu}^N \in \mathcal{O} \big) &\geq - \inf_{\mu \in \mathcal{O}} \mathcal{G}_T(\mu) , \label{LDP Lower Bound bar mu N finite T}
\end{align}
and we recall that $\mathcal{G}_T$ is defined in \eqref{eq: G T mu definition}.
\end{theorem}


For the rest of this Section, we focus on proving Theorem \ref{Lemma finite time intervals LDP}.

For an integer $m \gg 1$, we are going to define an approximate system $\lbrace \sigma^{(m),j}(t) \big\rbrace_{j\in I_N}$ with spatially-discretized interactions. 
To this end, for a positive integer $m$, let $\big\lbrace S^{(m)}_i \big\rbrace_{1\leq i \leq M_m} \subset \mathfrak{B}\big( \mathcal{E} \big)$ be disjoint measurable sets such that (i) $\text{diam}\big(S^{(m)}_i \big) \leq m^{-1}$, (ii) the interior of $S^{(m)}_i$ is nonempty and 
\begin{align}
\mathcal{E} =& \bigcup_{1\leq i \leq M_m} S^{(m)}_i .
\end{align} 
and (iv) $S^{(m)}_i$ is connected. Let $\theta^{(m)}_i$ be any point in $S^{(m)}_i$ and define $x^{(m),j}_N := \theta^{(m)}_l$ precisely when $x^j_N \in S^{(m)}_i$. 

It is assumed that the intensity of transition assumes the form, for $h \ll 1$,
\begin{align}
\mathbb{P}\big( \sigma^{(m),j}(t+h) = \alpha \; | \; \mathcal{F}_t \big) = h f_{\alpha}^{(m)}\big(x^j_N , \sigma^{(m),j}(t) , w^{(m)}(x^j_N , \sigma^{(m)}(t) ) \big)+ O(h^2)
\end{align}
where writing $S^{(m)}_i$ to be such that $x \in S^{(m)}_i$,
\begin{align}
 f_{\alpha}^{(m)}\big(x ,  \beta , \sigma(t) \big) =f_{\alpha}\big(\theta^{(m)}_i,\beta , w^{(m)}( x , \sigma(t) ) \big)
 \end{align}
and for any  $\theta \in S^{(m)}_i$ and any $\zeta \in \Gamma$,
\begin{align}
w^{(m)} (x , \sigma(t) ) =& \big( w^{(m)}_{\zeta}(x  , \sigma(t) ) \big)_{\zeta\in\Gamma} \\
w^{(m)}_{\zeta}(x , \sigma(t) ) =& N^{-1} \sum_{j=1}^{M_m} \sum_{k\in I_N : x^k_N \in S^{(m)}_i } \mathcal{J}(\theta^{(m)}_i , x^{(m),k}_N) \chi\big\lbrace \sigma^k(t) = \zeta \big\rbrace  . \label{eq: W s zeta definitoin}
\end{align}
For $(\alpha,\beta) \in \Xi$, define the empirical flux measure $\bar{\mu}^{(m),N}_{\alpha\mapsto\beta} \in \mathcal{M}\big( \mathcal{E}^{(m)} \times [0,\infty) \big)$ to be such that for measurable $A \subseteq \mathcal{E}$ and measurable $B \subseteq [0,\infty)$,
\begin{align}
\bar{\mu}_{\alpha\mapsto\beta }^{(m),N}(A \times B) = N^{-1} \sum_{j\in I_N} \chi\big\lbrace x^{ (m),j}_N \in A \big\rbrace \sum_{t \in B} \chi \big\lbrace \sigma^{(m),j}(t^-) = \alpha , \sigma^{(m),j}(t) = \beta  \big\rbrace ,
\end{align}
and write  $\bar{\mu}^{(m),N} = \big(\bar{\mu}^{(m),N}_{\alpha\mapsto\beta} \big)_{(\alpha,\beta) \in \Xi} \in  \mathcal{M}\big( \mathcal{E}^{(m)} \times [0,\infty) \big)^{\Xi}$. 
%
%
%

We next define a function $\mathcal{G}^{(m)}_T: \mathcal{M}\big( \mathcal{E}^{(m)} \times [0,T] \big)^{\Xi} \mapsto \mathbb{R}$ that will characterize the Large Deviation probabilities of the spatially-discretized process. 
Any  $\mu \in  \mathcal{M}\big( \mathcal{E} \times [0,\infty) \big)^{\Xi}$ must be of the form, for measurable $B \subseteq [0,\infty)$,
\begin{align}
\mu_{\alpha\mapsto \beta}( \theta^{(m)}_i \times B) =   \mu_{\alpha\mapsto\beta}^{(m,i)}(B) \label{eq: mu m decomposition}
\end{align}
for measures $\big( \mu_{\alpha\mapsto\beta}^{(m,i)} \big)_{i\leq M_m} \subset \mathcal{M}([0,\infty))$. Next, if for some $(\alpha,\beta) \in \Xi$, and some $i \leq M_m$, $\mu_{\alpha\mapsto\beta}^{(m,i)}$ does not have a density, then we define
\[
\mathcal{G}_T^{(m)}(\mu) = \infty.
\]
Otherwise, writing $p^{(m,i)}_{\alpha\mapsto\beta} :  [0,\infty) \mapsto [0,\infty)$ to be the density of $\mu_{\alpha\mapsto\beta}^{(m,i)}$, we define
 \begin{align}
\mathcal{G}^{(m)}_T (\mu) =  \sum_{ (\alpha , \beta) \in \Xi} \sum_{i=1}^{M_m}    \int_0^T   \ell\bigg(p^{(m,i)}_{\alpha \mapsto \beta}(x,t) / \lambda^{(m,i)}_{\alpha\mapsto\beta}( t) \bigg) \lambda^{(m,i)}_{\alpha\mapsto\beta}(t)   dt , \label{eq:  G (m) T definition}
 \end{align}
where
 \begin{align}
 \lambda^{(m,i)}_{\alpha\mapsto\beta}(t) =& f_{\beta}\big(\theta^{(m)}_i , \alpha , w^{(m,i)}(t) \big) \nu_t^{(m,i)}(\alpha)   \text{ where } \\
 w^{(m,i)}(t) =& \big( w^{(m,i)}_{\zeta}(t) \big)_{\zeta \in \Gamma} \\
 w^{(m,i)}_{\zeta}(t) =& \sum_{j=1}^{M_m} \mathcal{J}(\theta^{(m)}_i , \theta^{(m)}_j) \nu_t^{(m,j)}(\zeta)   \\
\nu_t^{(m,j)}(\zeta)  =&   \nu_0\big(\zeta \times S^{(m)}_j \big) + \int_0^t \bigg( \sum_{\beta: (\beta,\zeta) \in \Xi} p^{(m,j)}_{\beta \mapsto \zeta}(s) - \sum_{\beta: (\zeta,\beta) \in \Xi}  p^{(m,j)}_{\zeta \mapsto \beta}(s)  \bigg) ds .
 \end{align}
We can now characterize the Large Deviations of the above process with spatially-discretized interactions.
  \begin{theorem} \label{Large Deviations Spatially Discretized System}
 Let $\mathcal{A},\mathcal{O} \subset \mathcal{M}\big( \mathcal{E} \times [0,T] \big)^{\Xi}$ be (respectively) closed and open sets. Then
 \begin{align}
 \lsup{N} N^{-1} \log \mathbb{P}\big( \bar{\mu}_T^{(m),N} \in \mathcal{A} \big) &\leq - \inf_{\mu \in \mathcal{A}} \mathcal{G}_T^{(m)}(\mu) \\
  \linf{N} N^{-1} \log \mathbb{P}\big( \bar{\mu}_T^{(m),N}  \in \mathcal{O} \big) &\geq - \inf_{\mu \in \mathcal{O}} \mathcal{G}_T^{(m)}(\mu).
 \end{align}
 \end{theorem}
 \begin{proof}
 Let $ \lbrace y^{(m,i)}_{\zeta}(t) \rbrace_{i \leq M_m , \zeta \in \Gamma}$ be such that
 \begin{align}
 y^{(m,i)}_{\zeta}(t) = N^{-1} \sum_{j: x^j_N \in S^{(m)}_i} \chi\big\lbrace \sigma^{(m),j}(t) = \zeta \big\rbrace.
 \end{align}
Retracing our definitions, $ \lbrace y^{(m,i)}_{\zeta}(t) \rbrace_{i \leq M_m , \zeta \in \Gamma}$ constitute a Markovian system, with an intensity function
\begin{align}
U^{(m)}_{ \alpha\mapsto\beta}: \mathbb{R}^{M_m \times |\Gamma|} \mapsto [0,\infty)
\end{align}
such that for any $(\beta,\zeta) \in \Xi$,
\begin{multline}
\mathbb{P}\bigg(  y^{(m,i)}_{\zeta}(t + h) = y^{(m,i)}_{\zeta}(t ) + N^{-1} \text{ and } y^{(m,i)}_{\beta}(t + h) = y^{(m,i)}_{\beta}(t ) - N^{-1} \; \; \bigg|  \; \; \mathcal{F}_t  \bigg) \\ = h N  f_{\zeta}\big(\theta^{(m)}_i , \beta , w^{(m,i)}(t) \big) \nu_t^{(m,i)}(\beta) + O(h^2).
\end{multline}
The Large Deviations for the tracectory of $ \lbrace y^{(m,i)}_{\zeta}(t) \rbrace_{i \leq M_m , \zeta \in \Gamma}$ now follows from existing results on the Large Deviations of Chemical Reaction Networks. See for example \cite{Budhiraja2019,Barbet2023}.  \end{proof}
Our next task is to demonstrate that the spatially-discretized empirical flux $\bar{\mu}^{(m),N}_T$ is an excellent approximation to $\bar{\mu}^N_T$. First, we require that $\bar{\mu}^N_T$ inhabits a compact set with very high probability. To this end, for $T, l > 0$, define the constant $c_l$ to be such that
\begin{align}
\lsup{N} N^{-1} \log \mathbb{P}\bigg( \bar{\mu}^N_T \notin \hat{\mathcal{K}}_{T,l} \bigg) \leq - l, \label{eq: asymptotic K T l property}
\end{align}
where 
\begin{equation} \label{eq: K l T set}
\hat{\mathcal{K}}_{T,l} = \bigg\lbrace \mu \in \mathcal{M}\big( \mathcal{E} \times [0,T] \big)^{\Xi} \; :  \;  \sup_{(\alpha,\beta) \in \Xi } \mu_{\alpha \mapsto \beta} \big( \mathcal{E} \times [0,T] \big) \leq c_l \bigg\rbrace .
\end{equation}

\begin{lemma} \label{Lemma exponentially tight}
For any $T,l > 0$, there exists $c_l > 0$ such that \eqref{eq: K l T set} and \eqref{eq: asymptotic K T l property} hold.
\end{lemma}
The proof of Lemma \ref{Lemma exponentially tight} is standard (similar to the proof of Lemma \ref{Lemma K m bound}) and is neglected. 

Next, we  will compare the spatially-discretized system to the original system, via Girsanov's Theorem. Let $P^{(m),N}_{T} \in \mathcal{P}\big( \mathcal{D}([0,T], \Gamma)^N \big)$ be the probability law of the above system $\big( \sigma^{(m),j}(t)  \big)_{j\in I_N \fatsemi t\leq T}$.
Thanks to Girsanov's Theorem \cite{Bremaud1981},
\begin{equation}
\frac{dP^{N}_T}{dP^{(m),N}_T} (\sigma)= \exp\big( N \Upsilon^{(m),N}_T (\sigma) \big) \label{eq: Girsanov Expression Gamma T again}
\end{equation}
where $\Upsilon^{(m),N}_T: \mathcal{D}([0,T],\Gamma^N) \mapsto \mathbb{R}$ is such that
\begin{multline}
\Upsilon^{(m),N}_T(\sigma) = N^{-1} \sum_{j\in I_N} \int_0^T  \sum_{\beta : (\sigma^j(s),\beta) \in \Xi }\bigg\lbrace f_{\beta}\big(x^j_N , \sigma^j(s) ,  w^{(m)}(x^j_N, \sigma(s) ) \big) \\ - f_{\beta}(x^j_N , \sigma^j(s) , w(x^j_N, \sigma(s) ) \big) \bigg\rbrace ds \\ + N^{-1} \sum_{j\in I_N} \bigg[ \sum_{s \leq T \; : \; \sigma^j(s^-) \neq \sigma^j(s)}  \bigg\lbrace \log  f_{\sigma^j(s)}\big(x^j_N , \sigma^j(s^-) ,  w(x^j_N, \sigma(s^-) ) \big) \\ -  \log f_{\sigma^j(s)} \big(x^j_N , \sigma^j(s^-) ,  w^{(m)}(x^j_N, \sigma(s^-) ) \big) \bigg\rbrace \bigg]
\end{multline}
We now state the proof of Lemma \ref{Lemma finite time intervals LDP}.
\begin{proof}
The Large Deviation Principle is a consequence of Lemma \ref{Lemma exponentially tight} and Lemma \ref{Lemma Continuous Function Varadhans Integral Lemma} (further below), thanks to Bryc's Inverse Varadhan Lemma \cite{Dembo1998}. 
\end{proof}
 
 \begin{lemma} \label{Lemma Continuous Function Varadhans Integral Lemma}
 Let $\mathcal{H}: \mathcal{M}\big( \mathcal{E} \times [0,T] \big)^{\Xi} \mapsto \mathbb{R}$ be continuous and bounded. Then
 \begin{align}
 \lim_{N \mapsto \infty} N^{-1} \log \mathbb{E}\bigg[ \exp\bigg( N \mathcal{H}\big( \bar{\mu}^N_T \big) \bigg) \bigg] = \sup_{\mu \in  \mathcal{M}( \mathcal{E} \times[0,T] )^{\Xi} } \big\lbrace \mathcal{H}(\mu) - \mathcal{G}_T(\mu) \big\rbrace .
 \end{align}
  \end{lemma}
 \begin{proof}
 Let $c> 0$ be such that $|\mathcal{H}(\mu) | \leq c$. \newline
 
 \textit{Step 1: Upper Bound} \newline
 
 We start by demonstrating that
   \begin{align} \label{eq: lsup inequality in one direction}
\lsup{N} N^{-1} \log \mathbb{E}\bigg[ \exp\bigg( N \mathcal{H}\big( \bar{\mu}^N_T \big) \bigg) \bigg] \leq \sup_{\mu \in  \mathcal{M}( \mathcal{E} \times[0,T] )^{\Xi} } \big\lbrace \mathcal{H}(\mu) - \mathcal{G}_T(\mu) \big\rbrace .
 \end{align}
For the set $\hat{\mathcal{K}}_{T,l} $ defined in \eqref{eq: K l T set},
 \begin{multline}
\lsup{N} N^{-1} \log \mathbb{E}\bigg[ \exp\bigg( N \mathcal{H}\big( \bar{\mu}^N_T \big) \bigg) \bigg] \leq \\   \max \bigg\lbrace \lsup{N} N^{-1} \log \mathbb{E}\bigg[ \chi\big\lbrace \bar{\mu}^N_T \in \hat{\mathcal{K}}_{T,l}   \big\rbrace \exp\bigg( N \mathcal{H}\big( \bar{\mu}^N_T \big) \bigg) \bigg] \\
\; , \;  \lsup{N} N^{-1} \log \mathbb{E}\bigg[ \chi\big\lbrace \bar{\mu}^N_T \notin \hat{\mathcal{K}}_{T,l}  \big\rbrace \exp\bigg( N \mathcal{H}\big( \bar{\mu}^N_T \big) \bigg) \bigg]  \bigg\rbrace  
  \end{multline}
  Starting with the last term on the RHS, using the fact that $\mathcal{H}$ is upperbounded by $c$,
  \begin{align}
   \lsup{N}  N^{-1} \log \mathbb{E}\bigg[& \chi\big\lbrace \bar{\mu}^N_T \notin \hat{\mathcal{K}}_{T,l}  \big\rbrace \exp\bigg( N \mathcal{H}\big( \bar{\mu}^N_T \big) \bigg) \bigg]  \nonumber \\
   \leq &   \lsup{N} N^{-1} \log \mathbb{E}\bigg[ \chi\big\lbrace \bar{\mu}^N_T \notin \hat{\mathcal{K}}_{T,l}  \big\rbrace  \bigg] + c  
   \leq   -l + c
 \end{align}
 thanks to Lemma \ref{Lemma K m bound}. 
  
 Prokhorov's Theorem implies that $\hat{\mathcal{K}}_{T,l} $ is compact, and hence there exist elements $\lbrace \gamma^{(m)}_i \rbrace_{1\leq i \leq O_l} \subset \hat{\mathcal{K}}_{T,l} $ such that
 \begin{align}
\hat{\mathcal{K}}_{T,l}  \subseteq \bigcup_{i=1}^{O_l} B_{l}\big( \gamma^{(m)}_i \big)
 \end{align}
 where
 \begin{align}
 B_{l}\big( \gamma^{(m)}_i \big) = \big\lbrace \zeta \in  \mathcal{M}\big( \mathcal{E}^{(m)} \times [0,T] \big)^{\Xi} \; : \; d_{T} \big( \zeta , \gamma^{(m)}_i \big) < l^{-1} \big\rbrace .
 \end{align}
We thus obtain that
 \begin{multline}
  \lsup{N} N^{-1} \log \mathbb{E}\bigg[ \chi\big\lbrace \bar{\mu}^N_T \in  \hat{\mathcal{K}}_{T,l}  \big\rbrace \exp\bigg( N \mathcal{H}\big( \bar{\mu}^N_T \big) \bigg) \bigg] = \\
  \max_{i \leq O_l}\bigg\lbrace \lsup{N} N^{-1} \log \mathbb{E}\bigg[ \chi\big\lbrace \bar{\mu}^N_T \in  B_{m}\big( \gamma^{(m)}_i \big) \big\rbrace \exp\bigg( N \mathcal{H}\big( \bar{\mu}^N_T \big) \bigg) \bigg] \bigg\rbrace .
 \end{multline} 
  Define $\tilde{\mu}^{(m),N}_T \in \mathcal{M}\big( \mathcal{E} \times [0,T] \big)^{\Xi}$ to be such that, writing the components as \newline$\tilde{\mu}^{(m),N}_T = \big( \tilde{\mu}_{\alpha\mapsto\beta }^{(m),N} \big)_{(\alpha,\beta) \in \Xi}$,
 \begin{align}
\tilde{\mu}_{\alpha\mapsto\beta }^{(m),N}(A \times B) = N^{-1} \sum_{j\in I_N}  \chi\big\lbrace x^{  j}_N \in A \big\rbrace \sum_{t \in B} \chi \big\lbrace \sigma^{(m),j}(t^-) = \alpha , \sigma^{(m),j}(t) = \beta  \big\rbrace .
\end{align} 
 Furthermore, by Girsanov's Theorem,
 \begin{multline}
  \mathbb{E}\bigg[ \chi\big\lbrace \bar{\mu}^N_T \in  B_{m}\big( \gamma^{(m)}_i \big) \big\rbrace \exp\bigg( N \mathcal{H}\big( \bar{\mu}^N_T \big) \bigg) \bigg]  = \\   \mathbb{E}\bigg[ \chi\big\lbrace \tilde{\mu}^{(m),N}_T \in  B_{m}\big( \gamma^{(m)}_i \big) \big\rbrace \exp\bigg( N \mathcal{H}\big( \tilde{\mu}^{(m),N}_T \big) +  N \Upsilon^{(m)}_T\big(\sigma^{(m)} \big) \bigg) \bigg]   .
 \end{multline}   
  It is immediate from the definition that
  \begin{align}
  d_T\big( \tilde{\mu}^{(m),N}_T , \bar{\mu}^{(m),N}_T \big) &\leq \sup_{j\leq N} d_{\mathcal{E}}(x^j_N , x^{(m),j}_N) \\
  &\leq 2m^{-1}. \label{eq: bound difference tilde mu m N T}
  \end{align}
  Since $\mathcal{H}$ must be uniformly continuous over the compact set $\hat{\mathcal{K}}_{T,l}$, for any $\epsilon > 0$ for large enough $m$ it must be that
  \begin{align}
  \bigg|  \mathcal{H}\big( \bar{\mu}^{(m),N}_T \big) -  \mathcal{H}\big( \tilde{\mu}^{(m),N}_T \big) \bigg| \leq \epsilon.
  \end{align}
Furthermore Lemma \ref{Lemma Upsilon Difference epsilon } implies that as long as $m$ is large enough,
\begin{align}
  \Upsilon^{(m)}_T\big(\sigma^{(m)} \big) \leq \epsilon.
\end{align}
We thus find that, as long as $m$ is large enough,
  \begin{align*}
  \lsup{N} N^{-1} \log \mathbb{E}\bigg[ \chi\big\lbrace \bar{\mu}^{(m),N}_T \in  B_{m}\big( \gamma^{(m)}_i \big) \big\rbrace \exp\bigg( N \mathcal{H}\big( \bar{\mu}^{(m),N}_T \big) +  N \Upsilon^{(m)}_T\big(\sigma^{(m)} \big) \bigg) \bigg] \\
  \leq \lsup{N} N^{-1} \log \mathbb{E}\bigg[ \chi\big\lbrace \bar{\mu}^{(m),N}_T \in  B_{m}\big( \gamma^{(m)}_i \big) \big\rbrace \exp\bigg( N \mathcal{H}\big( \bar{\mu}^{(m),N}_T \big) +  2N \epsilon \bigg) \bigg] .
  \end{align*}
  Applying Varadhan's Integral Lemma \cite{Dembo1998} to the Large Deviations result of Lemma \ref{Large Deviations Spatially Discretized System} implies that
  \begin{multline}
   \lsup{N} N^{-1} \log \mathbb{E}\bigg[ \chi\big\lbrace \bar{\mu}^{(m),N}_T \in  B_{m}\big( \gamma^{(m)}_i \big) \big\rbrace \exp\bigg( N \mathcal{H}\big( \bar{\mu}^{(m),N}_T \big) +  2N \epsilon \bigg) \bigg] \\ =2 \epsilon + \sup_{\mu \in  B_{m}( \gamma^{(m)}_i )} \bigg\lbrace  \mathcal{H}\big(\mu) -  \mathcal{G}^{(m)}_T(\mu) \bigg\rbrace .
  \end{multline}
 Since we can take $\epsilon$ to be arbitrarily small, we therefore obtain that
 \begin{align}
 \lsup{N} N^{-1} \log \mathbb{E}\bigg[ \chi\big\lbrace \bar{\mu}^N_T \in & \hat{\mathcal{K}}_{T,l} \big\rbrace \exp\bigg( N \mathcal{H}\big( \bar{\mu}^N_T \big) \bigg) \bigg]\nonumber \\  &\leq \lim_{m\to\infty}  \sup_{i \leq O_m}\sup_{\mu \in  B_{m}( \gamma^{(m)}_i )} \bigg\lbrace  \mathcal{H}\big(\mu) -  \mathcal{G}^{(m)}_T(\mu) \bigg\rbrace \nonumber \\
 &=  \lim_{m\to\infty}  \sup_{ \mu \in \hat{\mathcal{K}}_{T,l} } \bigg\lbrace  \mathcal{H}(\mu) -  \mathcal{G}^{(m)}_T(\mu) \bigg\rbrace  \nonumber \\
  &\leq  \sup_{ \mu \in \hat{\mathcal{K}}_{T,l} } \bigg\lbrace  \mathcal{H}(\mu) -  \mathcal{G}_T(\mu) \bigg\rbrace ,
 \end{align}
 thanks to Lemma \ref{Lemma Upper Bound m convergence rate function}.
 
  \textit{Step 2: Lower Bound} \newline
 
 We now turn to proving that
 \begin{align}
\linf{N} N^{-1} \log \mathbb{E}\bigg[ \exp\bigg( N \mathcal{H}\big( \bar{\mu}^N_T \big) \bigg) \bigg] \geq \sup_{\mu \in  \mathcal{M}( \mathcal{E} \times[0,T] )^{\Xi}} \big\lbrace \mathcal{H}(\mu) - \mathcal{G}_T(\mu) \big\rbrace .
 \end{align}
 Let
 \begin{align}
 \tilde{\mathcal{M}}\big( \mathcal{E} \times  [0,T] \big)^{\Xi} 
\end{align}
 consist of all $(\mu_{\alpha\mapsto\beta})_{(\alpha,\beta) \in \Xi}$ such that (i) $\mu_{\alpha\mapsto\beta}$ has a density and (ii) the density of $\mu_{\alpha\mapsto\beta}$ is continuous. It is straightforward to check that for any {\color{blue} $\mu \in  \tilde{\mathcal{M}}\big( \mathcal{E} \times  [0,T] \big)^{\Xi} $} such that $\mathcal{G}_T(\mu) < \infty$, there exists a sequence $(\mu_i)_{i \geq 1}$ such that (i) $\mu_i \to \mu$ and (ii) $\mathcal{G}_T(\mu_i) \to \mathcal{G}_T(\mu)$. This means that
 \begin{align}
 \sup_{\mu \in  \mathcal{M}( \mathcal{E} \times[0,T] )^{\Xi}} \big\lbrace \mathcal{H}(\mu) - \mathcal{G}_T(\mu) \big\rbrace
 = \sup_{\mu \in  \tilde{\mathcal{M}}( \mathcal{E} \times[0,T] )^{\Xi}} \big\lbrace \mathcal{H}(\mu) - \mathcal{G}_T(\mu) \big\rbrace .
 \end{align} 
 Let $\xi \in  \tilde{\mathcal{M}}\big( \mathcal{E} \times [0,T] \big)^{\Xi} $ be an arbitrary measure such that 
  \begin{equation}
  \mathcal{G}_T(\xi) < \infty .
  \end{equation}
 Since $\xi$ is arbitrary, it suffices that we show that
  \begin{align} \label{eq: to establish lower bound xi 1}
  \lim_{\delta \to 0^+} \linf{N} N^{-1} \log  \mathbb{E}\bigg[ \chi\big\lbrace d_T\big( \bar{\mu}^N_T , \xi \big) < \delta \big\rbrace \exp\bigg( N \mathcal{H}\big( \bar{\mu}^N_T \big) \bigg) \bigg] \geq \mathcal{H}(\xi) - \mathcal{G}_T(\xi).
  \end{align}
  Since $\mathcal{H}$ is continuous, the variation of $\mathcal{H}$ over the set $\big\lbrace \mu \; : \;  d_T\big( \mu  , \xi \big) < \delta \big\rbrace$ gets arbitrarily small as $\delta \to 0$. In order that \eqref{eq: to establish lower bound xi 1} holds, it thus suffices that we prove that
    \begin{align}\label{eq: to establish lower bound xi 2}
  \lim_{\delta \to 0^+} \linf{N} N^{-1} \log  \mathbb{P}\bigg(  d_T\big( \bar{\mu}^N_T , \xi \big) < \delta \bigg) \geq   - \mathcal{G}_T(\xi).
  \end{align}
  Now
 \begin{align}\label{eq: to establish lower bound xi 2 restated}
 \mathbb{P}\bigg(  d_T\big( \bar{\mu}^N_T , \xi \big) < \delta \bigg) 
 =   \mathbb{E}\bigg[ \chi\big\lbrace  d_T\big( \tilde{\mu}^N_T , \xi \big) < \delta \big\rbrace \exp\bigg(  N \Upsilon^{(m)}_T\big(\sigma^{(m)} \big)\bigg) \bigg]  .
 \end{align}
  In light of \eqref{eq: to establish lower bound xi 2 restated}, in order that \eqref{eq: to establish lower bound xi 2} holds, it suffices to prove that
  \begin{align}
    \lim_{\delta \to 0^+} \linf{N} N^{-1} \log  \mathbb{P}\bigg(  d_T\big( \tilde{\mu}^N_T , \xi \big) < \delta \bigg) \geq  - \mathcal{G}_T(\xi).
  \end{align}
  In turn, as noted in \eqref{eq: bound difference tilde mu m N T}, it must hold that
    \begin{align}
  d_T\big( \tilde{\mu}^{(m),N}_T , \bar{\mu}^{(m),N}_T \big)  \leq 2m^{-1},
  \end{align}
 and therefore it suffices to prove that
  \begin{align} \label{eq: temporary m xi inequality}
   \lim_{m\to\infty} \lim_{\delta \to 0^+} \linf{N} N^{-1} \log  \mathbb{P}\bigg(  d_T\big( \bar{\mu}^{(m),N}_T , \xi \big) < \delta^2 \bigg)\geq  - \mathcal{G}_T(\xi).
   \end{align}
  Now the Large Deviations estimate in Lemma \ref{Large Deviations Spatially Discretized System} implies that 
      \begin{align}
  \lim_{m\to\infty} \lim_{\delta \to 0^+} \linf{N} N^{-1} \log  \mathbb{P}\bigg(  d_T\big( \bar{\mu}^{(m),N}_T , \xi \big) < \delta^2 \bigg) \nonumber \\
  \geq -  \lim_{m\to\infty} \lim_{\delta \to 0^+} \inf_{\mu : d_T(\mu,\xi) < \delta^2} \mathcal{G}^{(m)}_T(\mu).
  \end{align}
 As long as $m$ is large enough that $d_T\big( \xi , \xi^{(m)} \big) < \delta^2$, it must be that
 \begin{align}
 \inf_{\mu : d_T(\mu,\xi) < \delta^2} \mathcal{G}^{(m)}_T(\mu) \leq \mathcal{G}^{(m)}_T\big( \xi^{(m)} \big),
 \end{align}
 and therefore
 \begin{align}
  \lim_{m\to\infty} \lim_{\delta \to 0^+} \inf_{\mu : d_T(\mu,\xi) < \delta^2} \big\lbrace \mathcal{G}^{(m)}_T(\mu) \big\rbrace &\leq  \lim_{m\to\infty}  \mathcal{G}^{(m)}_T\big( \xi^{(m)} \big)   \nonumber \\
  &\leq  \mathcal{G}_T(\xi),
  \end{align}
  thanks to Lemma \ref{Lemma uniform convergence of rate function m}. We have established \eqref{eq: temporary m xi inequality}, and therefore \eqref{eq: to establish lower bound xi 2}, as required.
 
 \end{proof}

 \begin{lemma} \label{Lemma Upper Bound m convergence rate function}
Suppose that $\mathcal{H}: \mathcal{M}\big( \mathcal{E} \times [0,T] \big)^{\Xi} \mapsto \mathbb{R}$ is continuous and bounded. 
 Then,
 \begin{align}
\lim_{m\to\infty} \sup_{\mu \in   \mathcal{M}\big( \mathcal{E} \times [0,T] \big)^{\Xi} } \big\lbrace \mathcal{H}(\mu) - \mathcal{G}^{(m)}(\mu) \big\rbrace \geq
 \sup_{\mu \in  \mathcal{M}\big( \mathcal{E} \times [0,T] \big)^{\Xi}  } \big\lbrace \mathcal{H}(\mu) - \mathcal{G}(\mu) \big\rbrace .
 \end{align}
 \end{lemma}
 \begin{proof}
 Let $\xi^{(m)} \in \mathcal{M}\big(\mathcal{E}^{(m)} \times [0,T]\big)^{\Xi}$, be such that
 \begin{align}
 \mathcal{H}\big(\xi^{(m)} \big) - \mathcal{G}^{(m)}\big( \xi^{(m)} \big) \geq  \sup_{\mu \in \mathcal{K}} \big\lbrace \mathcal{H}(\mu) - \mathcal{G}^{(m)}(\mu) \big\rbrace - m^{-1}.
  \end{align}
  Since
  \[
   \sup_{\mu \in  \mathcal{M}\big( \mathcal{E} \times [0,T] \big)^{\Xi} } \big\lbrace \mathcal{H}(\mu) - \mathcal{G}^{(m)}(\mu) \big\rbrace \geq -c,
  \]
  and $\mathcal{H}$ is bounded, it must be the case that $\mathcal{G}^{(m)}\big(\xi^{(m)}\big) < \infty$. This means that it must be of the form, for measurable $B \subseteq [0,\infty)$,
\begin{align}
\xi^{(m)}_{\alpha\mapsto \beta}( \theta^{(m)}_i \times B) = \kappa(S^{(m)}_i) \xi_{\alpha\mapsto\beta}^{(m,i)}(B) \label{eq: xi m decomposition}
\end{align}
for measures $\big( \xi_{\alpha\mapsto\beta}^{(m,i)} \big)_{i\leq M_m} \subset \mathcal{M}([0,\infty))$, and $\xi^{(m,i)}_{\alpha\mapsto\beta}$ must have a density $p^{(m,i)}_{\alpha\mapsto\beta}:   [0,T] \mapsto \mathbb{R}$.  Define $\tilde{p}^{(m)}_{\alpha\mapsto\beta} : \mathcal{E} \times [0,\infty) \mapsto [0,\infty)$ to be the function such that for each $i \leq M_m$, for all $x\in S^{(m)}_i$,
\begin{align}
\tilde{p}^{(m)}_{\alpha\mapsto\beta}(x,t) = p^{(m,i)}_{\alpha\mapsto\beta}(t) ,
\end{align}
and let 
\begin{align}
\tilde{\mu}^{(m)}_{\alpha\mapsto\beta} \in \mathcal{M}\big( \mathcal{E} \times [0,T] \big)
\end{align}
be the measure with density $\tilde{p}^{(m)}_{\alpha\mapsto\beta}$. 

It follows from the definition in \eqref{eq:  G (m) T definition} that
\begin{align}
\mathcal{G}^{(m)}_T \big(\xi^{(m)} \big) =  \sum_{ (\alpha , \beta) \in \Xi}   \int_{\mathcal{E}} \int_0^T   \ell\bigg( \tilde{p}^{(m)}_{\alpha \mapsto \beta}(x,t) / \hat{\lambda}^{(m)}_{\alpha\mapsto\beta}(x, t) \bigg) \hat{\lambda}^{(m)}_{\alpha\mapsto\beta}(x,t)   dt \kappa(dx) , \label{eq:  G (m) T definition 2}
 \end{align}
where if $x\in S^{(m)}_i$,
 \begin{align*}
 \hat{\lambda}^{(m)}_{\alpha\mapsto\beta}(x,t) =& f_{\beta}\big(\theta^{(m)}_i , \alpha , \hat{w}^{(m)}(x,t) \big) \nu_t^{(m,i)}(\alpha,x)   \text{ where } \\
 \hat{w}^{(m)}(x,t) =& \big( \hat{w}^{(m)}_{\zeta}(x,t) \big)_{\zeta \in \Gamma} \\
 \hat{w}^{(m)}_{\zeta}(x,t) =& \sum_{j=1}^{M_m}  \mathcal{J}(\theta^{(m)}_i , \theta^{(m)}_j) \tilde{\nu}_t^{(m)}\big(\zeta, \mathcal{S}^{(m)}_j \big)   \\
 \tilde{\nu}_t^{(m)}(\zeta , x)  =&   \nu_0\big(\zeta \times S^{(m)}_j \big) + \int_0^t \bigg( \sum_{\beta: (\beta,\zeta) \in \Xi} \tilde{p}^{(m)}_{\beta \mapsto \zeta}(x,s) - \sum_{\beta: (\zeta,\beta) \in \Xi}  \tilde{p}^{(m)}_{\zeta \mapsto \beta}(x,s)  \bigg) ds .
 \end{align*}
Writing $\tilde{\mu}^{(m)} = \big( \tilde{\mu}^{(m)}_{\alpha\mapsto\beta} \big)_{(\alpha,\beta) \in \Xi}$, we compute that
 \begin{align}
\mathcal{G}_T \big(\tilde{\mu}^{(m)} \big) =  \sum_{ (\alpha , \beta) \in \Xi}     \int_0^T   \ell\bigg( \tilde{p}^{(m)}_{\alpha \mapsto \beta}(x,t) / \tilde{\lambda}^{(m)}_{\alpha\mapsto\beta}( x, t) \bigg) \tilde{\lambda}^{(m)}_{\alpha\mapsto\beta}(x,t)   dt \kappa(dx) ,
 \end{align}
where 
 \begin{align*}
\tilde{\lambda}^{(m)}_{\alpha\mapsto\beta}(x,t) =& f_{\beta}\big(x , \alpha , \tilde{w}^{(m)}(x,t) \big) \tilde{\nu}_t^{(m)}(\alpha,x)   \text{ where } \\
\tilde{w}^{(m)}(x,t) =& \big( \tilde{w}^{(m)}_{\zeta}(x,t) \big)_{\zeta \in \Gamma} \\
\tilde{w}^{(m)}_{\zeta}(x,t) =& \int_{\mathcal{E}}\mathcal{J}( x , y) \tilde{\nu}_t^{(m)}( \zeta , y) \kappa(dy)  .
 \end{align*}
 Comparing the definitions, it holds that if $x\in S^{(m)}_i$, then 
 \begin{align}
 \tilde{w}^{(m)}_{\zeta}(x,t) -  w^{(m,i)}_{\zeta}(t) = \sum_{j=1}^{M_m} \int_{S^{(m)}_i}\big( \mathcal{J}(x,y) - \mathcal{J}(\theta^{(m)}_i , \theta^{(m)}_j) \big) \kappa(dy).
 \end{align}
 Since $\mathcal{J}$ is continuous, it must therefore hold that
 \begin{align*}
\lim_{m\to\infty} \sup_{t\leq T}\sup_{\zeta\in\Gamma} \sup_{i \leq M_m} \sup_{x\in S^{(m)}_i} \big| \tilde{w}^{(m)}_{\zeta}(x,t) -  w^{(m,i)}_{\zeta}(t)  \big| = 0.
 \end{align*}
 This means that, necessarily,
 \begin{align}
\lim_{m\to\infty} \big| \mathcal{G}_T \big(\tilde{\mu}^{(m)} \big)  -  \mathcal{G}^{(m)}_T \big(\xi^{(m)} \big) \big| = 0,
 \end{align}
 and hence
 \begin{align}
 \lim_{m\to\infty} \sup_{t\leq T} \sup_{\zeta\in\Gamma} \sup_{i \leq M_m} \sup_{x\in S^{(m)}_i}\big| \tilde{\lambda}^{(m)}_{\alpha\mapsto\beta}(x,t)  -  \hat{\lambda}^{(m)}_{\alpha\mapsto\beta}(x,t)  \big| = 0.
 \end{align}
 This in turn implies that
  \begin{align*}
 \lim_{m\to\infty} \sup_{t\leq T} \sup_{\zeta\in\Gamma} \sup_{i \leq M_m} \sup_{x\in S^{(m)}_i}\bigg|   \ell\bigg( \tilde{p}^{(m)}_{\alpha \mapsto \beta}(x,t) / \hat{\lambda}^{(m)}_{\alpha\mapsto\beta}(x, t) \bigg) - \ell\bigg( \tilde{p}^{(m)}_{\alpha \mapsto \beta}(x,t) / \tilde{\lambda}^{(m)}_{\alpha\mapsto\beta}( x, t) \bigg)  \bigg| = 0,
 \end{align*}
 which implies the Lemma.
 \end{proof}
\begin{lemma} \label{Lemma uniform convergence of rate function m}
Suppose that (i) $\mathcal{G}_T(\mu)  < \infty$ and that (ii) for each $(\alpha,\beta) \in \Xi$, $\mu_{\alpha\mapsto \beta}$ has a density $p_{\alpha\mapsto\beta}: \mathcal{E} \times [0,T] \mapsto [0,\infty)$, and that (iii) $(x,t) \mapsto p_{\alpha\mapsto\beta}(x,t)$ is continuous. 
Then
\begin{align}
\lim_{m\to\infty} \mathcal{G}^{(m)}\big( \mu^{(m)} \big) = \mathcal{G}(\mu).\label{eq: G m convergence}
\end{align}
\end{lemma}
\begin{proof}
Following the definition in \eqref{eq:  G (m) T definition},
\begin{align}
 \mathcal{G}_T^{(m)}\big( \mu^{(m)} \big) =    \sum_{ (\alpha , \beta) \in \Xi} \sum_{i=1}^{M_m}    \int_0^T   \ell\bigg(p^{(m,i)}_{\alpha \mapsto \beta}(x,t) / \lambda^{(m,i)}_{\alpha\mapsto\beta}( t) \bigg) \lambda^{(m,i)}_{\alpha\mapsto\beta}(t)   dt ,
\end{align}
where 
\begin{equation}
p^{(m,i)}_{\alpha\mapsto\beta}(t) = \int_{S^{(m)}_i} p_{\alpha\mapsto\beta}(x,t) \kappa(dx).
\end{equation}
and
 \begin{align}
 \lambda^{(m,i)}_{\alpha\mapsto\beta}(t) =& f_{\beta}\big(\theta^{(m)}_i , \alpha , w^{(m,i)}(t) \big) \nu_t^{(m,i)}(\alpha)   \text{ where } \\
 w^{(m,i)}(t) =& \big( w^{(m,i)}_{\zeta}(t) \big)_{\zeta \in \Gamma} \\
 w^{(m,i)}_{\zeta}(t) =& \sum_{j=1}^{M_m} \mathcal{J}(\theta^{(m)}_i , \theta^{(m)}_j) \nu_t^{(m,j)}(\zeta)   \\
\nu_t^{(m,j)}(\zeta)  =&   \nu_0\big(\zeta \times S^{(m)}_j \big) + \int_0^t \bigg( \sum_{\beta: (\beta,\zeta) \in \Xi} p^{(m,j)}_{\beta \mapsto \zeta}(s) - \sum_{\beta: (\zeta,\beta) \in \Xi}  p^{(m,j)}_{\zeta \mapsto \beta}(s)  \bigg) ds
 \end{align}
 Similarly, 
 \begin{align}
\mathcal{G}_T(\mu) =  \sum_{ (\alpha , \beta) \in \Xi} \int_{\mathcal{E}}    \int_0^T    \ell\big(p_{\alpha \mapsto \beta}(x,t) / \lambda_{\alpha\mapsto\beta}(x,t) \big) \lambda_{\alpha\mapsto\beta}(x,t)   dt  \kappa(dx).
\end{align}
where
 \begin{align}
\lambda_{\alpha\mapsto\beta}(x,t) =& f_{\beta}(x,\alpha, w (x,t) ) \nu_t(\alpha,x) \label{eq: lambda alpha beta x t} \\
\nu_t(\alpha,x) =& \nu_0(\alpha,x) + \int_0^t \bigg( \sum_{\beta: (\beta,\alpha) \in \Xi} p_{\beta \mapsto \alpha}(x,s) - \sum_{\beta: (\alpha,\beta) \in \Xi} p_{\alpha \mapsto \beta}(x,s) \bigg) ds \\
w(x,t) =& \big(w_{\zeta}(x,t) \big)_{\zeta\in\Gamma} \\
w_{\zeta}(x,t) =&\int_{\mathcal{E}}   \mathcal{J}(x,y)\nu_t(\zeta,y) \kappa(dy)
\end{align}
Since $p_{\alpha\mapsto\beta}$ is continuous, it follows that 
\begin{align} \label{eq: p m uniform convergence}
\lim_{m\to\infty} \sup_{t\leq T} \sup_{(\alpha,\beta) \in \Xi}  \sup_{i \leq M_m} \sup_{x \in S^{(m)}_i} \bigg| p_{\alpha\mapsto\beta}(x,t) - \kappa\big(S^{(m)}_i\big)^{-1} p_{\alpha\mapsto\beta}^{(m,i)}(t) \bigg|= 0.
\end{align}
\eqref{eq: p m uniform convergence} also implies that 
\begin{align}\label{eq: nu t uniform convergence}
\lim_{m\to\infty} \sup_{t\leq T} \sup_{\alpha \in \Gamma}  \sup_{i \leq M_m} \sup_{x \in S^{(m)}_i} \bigg| \nu_t(\alpha,x)  - \kappa\big(S^{(m)}_i\big)^{-1}  \nu_t^{(m,i)}(\alpha)   \bigg|= 0.
\end{align}
\eqref{eq: p m uniform convergence} and the uniform continuity of $(x,y) \mapsto \mathcal{J}(x,y)$ imply that
\begin{align} \label{eq: uniform convergence of w m}
\lim_{m\to\infty} \sup_{t\leq T}   \sup_{\alpha \in \Gamma}   \sup_{i \leq M_m} \sup_{x \in S^{(m)}_i} \bigg| w_{\alpha}(x,t) -  w^{(m,i)}_{\alpha}(t)  \bigg| = 0.
\end{align}
Since $f$ is Lipschitz in its third argument, \eqref{eq: p m uniform convergence}, \eqref{eq: nu t uniform convergence} and \eqref{eq: uniform convergence of w m}  imply that
 \begin{align}
 \lim_{m\to\infty} \sup_{t\leq T}   \sup_{\alpha \in \Gamma}   \sup_{i \leq M_m} \sup_{x \in S^{(m)}_i} \bigg| \lambda_{\alpha\mapsto\beta}(x,t) -  \kappa\big(S^{(m)}_i\big)^{-1} \lambda^{(m,i)}_{\alpha\mapsto\beta}(t)   \bigg| = 0.
 \end{align}
 Since $a \mapsto \ell(a)$ is continuous, we can therefore conclude that \eqref{eq: G m convergence} holds.
\end{proof}

\begin{lemma} \label{Lemma Upsilon Difference epsilon }
For any $\epsilon ,T,l> 0$, there exists $\mathfrak{m}_{\epsilon,T,l}$ such that for all $m \geq \mathfrak{m}_{\epsilon,T,l}$, if $\bar{\mu}^N \in \mathcal{K}_{T,l}$ then necessarily
\begin{equation}
 \big| \Upsilon^{(m),N}_T(\sigma )  \big| \leq \epsilon .
 \end{equation}
\end{lemma}
\begin{proof}
Recall that $f$ is bounded from below by a strictly positive constant 
The fact that $f$ is Lipschitz and means that there is a universal constant $C$ such that
\begin{multline}
\bigg| \log  f_{\sigma^j(s)}\big(x^j_N , \sigma^j(s^-) ,  w(x^j_N, \sigma(s) ) \big) -  \log f_{\sigma^j(s)} \big(x^j_N , \sigma^j(s^-) ,  w^{(m)}(x^j_N, \sigma(s) ) \big) \bigg| \\ \leq C \big| w(x^j_N, \sigma(s) ) - w^{(m)}(x^j_N, \sigma(s) ) \big|.
\end{multline}
In turn, the uniform convergence of the interactions implies that for any $\delta > 0$, one can find large enough $m$ such that for all $s \geq 0$,
\begin{align}
 \big| w(x^j_N, \sigma(s) ) - w^{(m)}(x^j_N, \sigma(s) ) \big| \leq \delta .
\end{align}
We can therefore conclude that for arbitrary $\tilde{\delta} > 0$, for all large enough $m$ it holds that 
\begin{align}
\big| \Upsilon^{(m),N}_T(\sigma) \big| &\leq \tilde{\delta} \bigg( 1 + N^{-1}  \sum_{(\alpha,\beta)\in \Xi} \sum_{j\in I_N} Z^j_{\alpha\mapsto\beta}(T)  \bigg) \\
&\leq  \tilde{\delta} \big( 1 + l  \big)
\end{align}
as long as $\bar{\mu}^N \in \mathcal{K}_{T,l}$. Thus as long as $\tilde{\delta}(1+ l) < \epsilon$, the Lemma will hold.
 \end{proof}

\section{An Application: Transition Paths For Hawkes Models in Epidemiology and Neuroscience}
 
We consider a simple stochastic SIS model for a structured population. There is certainly much work that has determined the large $N$ limiting dynamics for these models \cite{Britton2019}. However to the knowledge of these authors, there does not exist a spatially-extended {\color{blue} Large Deviation Principle} in the manner of this paper. The computation of optimal transition paths for spatially extended systems has become of increasing interest in recent years \cite{Fleurantin2023,Bernuzzi2024,Avitabile2024}.

We first outline a model of $N \gg 1$ people on a structured network. The nodes of the network reside in a domain $\mathcal{E} = \mathbb{S}^1$. The position of the $j^{th}$ person is $x^j =2\pi  j/N$, and their state is $\sigma^j(t) \in \lbrace S, I \rbrace$ (i.e. susceptible or infected). The probability of a positive connection from $\alpha \mapsto \theta$ is 
$\mathcal{J}(\theta,\alpha)$: i.e. $\mathbb{P}\big( K^{jk} = 1 \big) = \phi_N \mathcal{J}(x^j_N , x^k_N)$, and $\mathbb{P}\big( K^{jk} = 0 \big) =1- \phi_N \mathcal{J}(x^j_N , x^k_N)$. We assume symmetric connections, so that $\mathcal{J}(\theta,\alpha) = \mathcal{J}(\alpha,\theta)$. We also assume that $\mathcal{J}$ is piecewise continuous. 


The probability that a susceptible person transitions to being infected over the time interval $[t, t+h]$ (For $h \ll 1$) is, for a positive parameter $\beta$,
\begin{align}
h \beta W^j(t) + O(h^2) \text{  where  } 
W^j(t) = N^{-1}\sum_{k \in I_N}J^{jk}\chi\lbrace \sigma^k(t) = I \rbrace
\end{align}
The probability that an infected person transitions back to being susceptible over a time interval $[t , t+h]$ is constant, i.e. for some $\alpha > 0$ it is
\begin{align}
\alpha h + O(h^2).
\end{align}
Write $s(\theta,t) \in [0,1]$ to represent the proportion of susceptible people at position $\theta \in \mathcal{E}$ at time $t$ in the large $N$ limit, and let $i(\theta,t) \in [0,1]$ be the proportion of infected people. Since these are the only two possibilities, it must be that
\begin{align}
s(\theta,t) + i(\theta,t) = 1.
\end{align}
{\color{blue} Let us } first write out the large $N$ limiting dynamics. This is non-stochastic, and such that
\begin{align} \label{eq: large N limiting dynamics}
\frac{ds(t,\theta)}{dt} = - \beta s(t,\theta) \int_{\mathcal{E}} \mathcal{J}(\theta,\tilde{\theta}) (1-s(t,\tilde{\theta})) d\tilde{\theta} + \alpha (1 - s(t,\theta))
\end{align}

The Large Deviations Rate function governing the proportion of susceptible people is $\mathcal{H}_T: H^1 \big( \mathcal{E} \times  [0,T] , [0,\infty)\big) \to \mathbb{R}$ assumes the following form.  The time derivative of $s$ is written as $\dot{s}(t,\theta)$.

The rate function assumes the form
\begin{align}
\mathcal{H}_T(s) =  \int_0^T  \int_{\mathcal{E}} L_{\theta}(\dot{s}(t,\theta), s(t)) dt d\theta 
\end{align}
where for any $\theta \in \mathcal{E}$, we define
\[
L_{\theta}: \mathbb{R} \times  \mathcal{C}(\mathcal{E}) \to \mathbb{R}
\]
is defined to be such that
\begin{align}
L_{\theta}(\dot{s},s) =& \inf\big\lbrace  \ell\big(a / \lambda_{\theta} (s) \big)\lambda_{\theta} (s) + \alpha \ell\big( b / \lbrace \alpha (1-s(\theta)) \rbrace \big) (1-s(\theta)) \nonumber \\  &\text{ where }a,b \geq 0 \text{ and } \dot{s} =b-a  \big\rbrace \label{eq: L theta u u} \\
\lambda_{\theta}:& \mathcal{C}(\mathcal{E}) \to \mathbb{R} \text{ is such that }\\
\lambda_{\theta}(s) =& \beta s(\theta) \int_{\mathcal{E}}\mathcal{J}(\theta,\tilde{\theta})\big(1- s(\tilde{\theta})\big) d\tilde{\theta} \text{ and }\\
\ell(x) =& x\log x -x + 1.
\end{align}
We note that $L_{\theta}(\dot{s},s)$ is uniquely minimized for the large $N$ limiting dynamics, i.e. $L_{\theta}(\dot{s},s) = 0$ if and only if
\[
\dot{s} = - \beta s(t,\theta) \int_{\mathcal{E}} \mathcal{J}(\theta,\tilde{\theta}) (1-s(t,\tilde{\theta})) d\tilde{\theta} + \alpha (1 - s(t,\theta)).
\]
In computing the Large Deviations rate function for trajectories that differ from the above, we are trying to understand the relative likelihood of rare noise-induced events that differ from the above dynamics.

It turns out that  the infimum in \eqref{eq: L theta u u} is uniquely realized. We note this in the following lemma.
\begin{lemma} \label{infimum using calculus}
For $s \in \mathcal{C}(\mathcal{E})$ and $\theta \in \mathcal{E}$, and any $\dot{s} \in \mathbb{R}$,
\begin{align}
L_{\theta}(\dot{s}, s ) = \alpha(1- s(\theta)) \ell \bigg(  \frac{ \lambda_{\theta}(s)}{A_\theta(\dot{s},s)}\bigg) + \lambda_{\theta}(s)  \ell \bigg(  \frac{A_\theta(\dot{s},s)}{ \lambda_{\theta}(s)}  \bigg)   ,
\end{align}
where $A_{\theta}: \mathbb{R} \times \mathcal{C}(\mathcal{E}) \to\mathbb{R}$ is such that
\begin{align}
A_{\theta}(\dot{s} , s) =  \frac{1}{2}\bigg( -\dot{s} + \big( \dot{s}^2 + 4\alpha \lambda_{\theta}(s) (1- s(\theta)) \big)^{1/2} \bigg)
\end{align}
\end{lemma}
\begin{proof}
Fixing $\dot{s}, s$, this is effectively a 1d optimization problem (fixing $b = \dot{s} + a$) of the function
\[
\tilde{L}(a) := \ell\big(a / \lambda_{\theta} (s) \big)\lambda_{\theta} (s) + \alpha \ell\big( (\dot{s} + a) / \lbrace \alpha (1-s(\theta)) \rbrace \big) (1-s(\theta)),
\]
with domain $a \geq \max \big\lbrace 0, -\dot{s} \big\rbrace$. Since $\ell$ is convex, it must be that $a \mapsto \tilde{L}(a)$ is convex, and the infimum must occur at points such that $\partial_a \tilde{L}_a = 0$. We differentiate and find that the optimal $a$ must be such that
\begin{align}
\log\bigg( \frac{a}{\lambda_{\theta}(s)} \bigg) + \log\bigg( \frac{ \dot{s}(\theta) + a}{ \alpha(1-s(\theta))} \bigg) = 0.
\end{align}
This means that
\begin{align}
a ( \dot{s}(\theta) + a) = \alpha \lambda_{\theta}(s) (1-s(\theta))
\end{align}
and therefore
\begin{align}
a^2 + a \dot{s}(\theta) -  \alpha\lambda_{\theta}(s)(1-s(\theta)) = 0.
\end{align}
Since $a \geq 0$, the only valid root is
\begin{align}
a = \frac{1}{2}\bigg( -\dot{s}(\theta) + \big( \dot{s}(\theta)^2 + 4\alpha \lambda_{\theta}(s) (1- s(\theta)) \big)^{1/2} \bigg)
\end{align}
\end{proof}
\begin{lemma}
For every $\theta \in \mathcal{E}$ and $s \in \mathcal{C}(\mathcal{E})$, the function $\dot{s} \mapsto L_{\theta}(\dot{s},s)$ is strictly convex.
\end{lemma}
\begin{proof}
First, it is proved in Lemma \ref{infimum using calculus} that the infimum in \eqref{eq: L theta u u} is always realized at a unique $a \geq 0$. Consider $\dot{s}_1 , \dot{s}_2 \in \mathbb{R}$ and suppose that for some $\zeta \in [0,1]$, $\dot{s} = \zeta \dot{s}_1 + (1-\zeta)\dot{s}_2$. Let $a_1,a_2 \geq  0$ be the (respective) values of $a$ that realize the infimum, i.e. they are such that
\begin{align}
L_{\theta}(\dot{s}_1,s) =&  \ell\big(a_1 / \lambda_{\theta} (s) \big)\lambda_{\theta} (s) + \alpha \ell\big( (\dot{s}_1 + a_1) / \lbrace \alpha (1-s(\theta)) \rbrace \big) (1-s(\theta)) \\
L_{\theta}(\dot{s}_2,s) =&  \ell\big(a_2 / \lambda_{\theta} (s) \big)\lambda_{\theta} (s) + \alpha \ell\big( (\dot{s}_2 + a_2) / \lbrace \alpha (1-s(\theta)) \rbrace \big) (1-s(\theta)) \\
\dot{s}_1 + a_1 \geq &0 \\
\dot{s}_2 + a_2 \geq &0.
\end{align}
Write $a = \zeta a_1 + (1-\zeta)a_2$, and notice that $\dot{s} + a \geq 0$.  If we substitute $a$ into the RHS of \eqref{eq: L theta u u}, then since the function $\ell$ is strictly convex,
\begin{align}
L_{\theta}(\dot{s} , s) \leq  &\ell\big(a / \lambda_{\theta} (s) \big)\lambda_{\theta} (s) + \alpha \ell\big( (\dot{s} + a ) / \lbrace \alpha (1-s(\theta)) \rbrace \big) (1-s(\theta)) \nonumber \\  < &\zeta  \ell\big(a_1 / \lambda_{\theta} (s) \big)\lambda_{\theta} (s) + (1-\zeta) \ell\big(a_2 / \lambda_{\theta} (s) \big)\lambda_{\theta} (s) \nonumber \\
&\zeta \alpha \ell\big( (\dot{s}_1 + a_1) / \lbrace \alpha (1-s(\theta)) \rbrace \big) (1-s(\theta)) \nonumber \\ &+ (1-\zeta)  \alpha \ell\big( (\dot{s}_2 + a_2) / \lbrace \alpha (1-s(\theta)) \rbrace \big) (1-s(\theta)) \nonumber \\
=& \zeta L_{\theta}(\dot{s}_1,s) + (1-\zeta) L_{\theta}(\dot{s}_2,s).
\end{align}

\end{proof}

\subsection{Euler-Lagrange Equations for the Optimal Trajectory}

Fix an initial distribution of population $\bar{s}_0 \in \mathcal{C}(\mathcal{E})$ and a final population distribution $\bar{s}_T \in  \mathcal{C}(\mathcal{E})$. Assume that
\begin{align}
\inf_{\theta\in\mathcal{E}} \lbrace \bar{s}_0(\theta) , \bar{s}_T(\theta) \rbrace &> 0 \\
\sup_{\theta\in\mathcal{E}} \lbrace \bar{s}_0(\theta) , \bar{s}_T(\theta) \rbrace   &< 1.
\end{align}
Our main result in this section is that any optimal trajectory must satisfy the following Euler-Lagrange equations. Unfortunately, in general there will not be a unique solution to these equations. See for instance \cite{Heyman2008,Zakine2023} for more details on how to compute the optimal path numerically.
\begin{theorem}\label{Theorem Second Order Integro Differential Equation}
Suppose that $s \in \mathcal{C}([0,T], \mathcal{C}(\mathcal{E}))$ is such that
\begin{align} \label{eq: s optimal criterion}
 \mathcal{H}_T(s) = \inf\big\lbrace \mathcal{H}_T(u) \; : \; u_0 = \bar{s}_0 \text{ and }u_T = \bar{s}_T \big\rbrace
\end{align}
and for each $\theta \in \mathcal{E}$, $t\mapsto s_t(\theta)$ is twice continuously differentiable, with first and second derivatives written (respectively) as $\dot{s}_t(\theta)$ and $\ddot{s}_t(\theta)$. Any minimizer must satisfy the second-order integro-differential equation, for all $\theta \in \mathcal{E}$ and $t\in [0,T]$,
\begin{align}
\ddot{s}_t(\theta)\frac{\partial^2 L_{\theta}}{\partial \dot{s}^2} (\dot{s}, s )   + \dot{s}_t(\theta) \mathcal{O}_{\theta}(\dot{s}_t,s_t) = \mathcal{G}_{\theta}(\dot{s}_t ,s_t)
\end{align}
and $\mathcal{G}_{\theta}, \mathcal{O}_{\theta}: \mathcal{C}(\mathcal{E}) \times \mathcal{C}(\mathcal{E}) \mapsto \mathbb{R}$ are bounded nonlocal smooth operators defined in the course of the proof. Furthermore (since $L_{\theta}$ is convex in its first argument)
\begin{align}
\frac{\partial^2 L_{\theta}}{\partial \dot{s}^2} (\dot{s}, s ) >  0.
\end{align}
\end{theorem}
\begin{lemma}
There is a unique $s$ satisfying \eqref{eq: s optimal criterion}.
The optimal trajectory is such that at each $\theta \in \mathcal{E}$,
\begin{align} \label{eq: Calculus of Variations identity}
\frac{d}{dt} \frac{\partial L_{\theta}}{\partial \dot{s}} (\dot{s}(t,\theta), s ) = \mathcal{G}_{\theta}(\dot{s},s)
\end{align}
where
\begin{align}
\mathcal{G}_{\theta}:& \mathcal{C}(\mathcal{E}) \times \mathcal{C}(\mathcal{E}) \to \mathbb{R} 
\end{align}
is defined to be such that for any $x \in L^2(\mathcal{E})$,
\begin{align}
\lim_{\epsilon \to 0^+} \epsilon^{-1} \int_{\mathcal{E}} \big(  L_{\theta}(\dot{s}(\theta), s +\epsilon x)  - L_{\theta}(\dot{s}(\theta), s)   \big) d\theta  = \int_{\mathcal{E}}x(\theta) \mathcal{G}_{\theta}(\dot{s},s) d\theta.
\end{align} 
\end{lemma}
\begin{proof}
The fact that the infimum is realized follows from the fact that $\mathcal{H}_T$ is lower semi-continuous. The identity in \eqref{eq: Calculus of Variations identity} is a standard result from Calculus of Variations.
\end{proof}
 We now compute an expression for $\mathcal{G}_{\theta}$. To this end, let $D\lambda_{\theta}(s)\cdot x$ be the Frechet Derivative of $\lambda_{\theta}$ in the direction $x\in L^2(\mathcal{E})$, i.e.
\begin{align}
D\lambda_{\theta}(s)\cdot x= x(\theta) \beta \int_{\mathcal{E}}\mathcal{J}(\theta,\tilde{\theta})(1- s(\tilde{\theta})) d\tilde{\theta} - \beta s(\theta) \int_{\mathcal{E}} \mathcal{J}(\theta,\tilde{\theta}) x(\tilde{\theta}) d\tilde{\theta}.
\end{align} 
and let $DA_{\theta}(\dot{s},s)\cdot x$ be the Frechet Derivative of $A_{\theta}$ in the direction $x\in L^2(\mathcal{E})$, i.e.
\begin{align}
DA_{\theta}(\dot{s},s)\cdot x = \alpha \big( \dot{s}^2 + 4\alpha \lambda_{\theta}(s) (1- s(\theta)) \big)^{-1/2}\big( - \lambda_{\theta}(s)x(\theta) + (1-s(\theta)) D\lambda_{\theta}(s)\cdot x \big)
\end{align} 
We next compute the partial derivatives with respect to $\dot{s}$.
\begin{lemma}
\begin{align*}
\frac{\partial  L_{\theta}}{\partial \dot{s}}(\dot{s}, s ) =& -\frac{\partial A_{\theta}}{\partial \dot{s}}(\dot{s},s)  \log\bigg( \frac{\lambda_{\theta}(s)}{A_{\theta}(\dot{s},s)} \bigg)  \bigg\lbrace 1 + \frac{\alpha(1-s(\theta))\lambda_{\theta}(s)}{A_{\theta}(\dot{s},s)^2}\bigg\rbrace  \\
\frac{\partial^2 L_{\theta}}{\partial \dot{s}^2} (\dot{s}, s ) =& -\frac{\partial^2 A_{\theta}}{\partial \dot{s}^2}(\dot{s},s)  \log\bigg( \frac{\lambda_{\theta}(s)}{A_{\theta}(\dot{s},s)} \bigg)  \bigg\lbrace 1+ \frac{\alpha(1-s(\theta))\lambda_{\theta}(s)}{A_{\theta}(\dot{s},s)^2}\bigg\rbrace  \nonumber \\
&+\bigg(\frac{\partial A_{\theta}}{\partial \dot{s}}(\dot{s},s) \bigg)^2 \bigg\lbrace A_{\theta}(\dot{s},s)^{-1} + \frac{2\alpha (1-s(\theta)) \lambda_{\theta}(s)}{A_\theta(\dot{s},s)^3} \log \bigg( \frac{\lambda_{\theta}(s)}{A_\theta(\dot{s},s)} \bigg)\nonumber \\ &+ \frac{\alpha (1-s(\theta)) \lambda_{\theta}(s) }{A_\theta(\dot{s},s)^3} \bigg\rbrace \\
\frac{\partial A_{\theta}}{\partial \dot{s}}(\dot{s},s) =& - \frac{1}{2} + \frac{\dot{s}}{2} \big( \dot{s}^2 + 4\alpha \lambda_{\theta}(s) (1- s(\theta)) \big)^{-1/2} \\
\frac{\partial^2 A_{\theta}}{\partial \dot{s}^2}(\dot{s},s) =&  \frac{1}{2} \big( \dot{s}^2 + 4\alpha \lambda_{\theta}(s) (1- s(\theta)) \big)^{-1/2} -   \frac{\dot{s}^2}{2} \big( \dot{s}^2 + 4\alpha \lambda_{\theta}(s) (1- s(\theta)) \big)^{-3/2} 
\end{align*}
\end{lemma}
\begin{lemma}
\begin{equation*}
\mathcal{G}_\theta(\dot{s},s) = \mathcal{N}_{\theta}(\dot{s},s) + \beta \mathcal{M}_\theta(s)  \int_{\mathcal{E}}\mathcal{J}(\theta,\tilde{\theta})(1- s(\tilde{\theta})) d\tilde{\theta} - \beta  \int_{\mathcal{E}} \mathcal{M}_{\tilde{\theta}}(s) \mathcal{J}(\tilde{\theta},\theta) s(\tilde{\theta}) d\tilde{\theta}
\end{equation*}
where  $\mathcal{M}_{\theta} , \mathcal{N}_{\theta}: \mathcal{C}(\mathcal{E}) \times \mathcal{C}(\mathcal{E}) \to \mathbb{R}$ are such that
\begin{multline} \label{eq: M dot s s definition}
\mathcal{M}_{\theta}(\dot{s},s) =   \ell \bigg(  \frac{A_\theta(\dot{s}(\theta),s)}{ \lambda_{\theta}(s)}  \bigg) + \bigg( \frac{\alpha(1- s(\theta))} {A_{\theta}(\dot{s},s)} + \frac{A_{\theta}(\dot{s},s)} {\lambda_{\theta}(s)} \bigg)\log \bigg(  \frac{ \lambda_{\theta}(s)}{A_\theta(\dot{s},s)}\bigg)  \\
+   \alpha \big(1-s(\theta)\big)  \big( \dot{s}(\theta)^2 + 4\alpha \lambda_{\theta}(s) (1- s(\theta)) \big)^{-1/2}    \log \bigg(  \frac{ \lambda_{\theta}(s)}{A_\theta(\dot{s},s)}\bigg) \times \bigg\lbrace  -\alpha(1-s(\theta)) \frac{\lambda_{\theta}(s)}{A_{\theta}(\dot{s}(\theta),s)^2} -1  \bigg\rbrace
\end{multline}
and
\begin{multline}
\mathcal{N}_{\theta}(\dot{s},s)  =  -\alpha \ell \bigg(  \frac{ \lambda_{\theta}(s)}{A_\theta(\dot{s},s)}\bigg)  \\ + \alpha \lambda_{\theta}(s) \big( \dot{s}^2 + 4\alpha \lambda_{\theta}(s) (1- s(\theta)) \big)^{-1/2}  \log \bigg(  \frac{ \lambda_{\theta}(s)}{A_\theta(\dot{s},s)}\bigg) \bigg\lbrace \alpha(1-s(\theta)) \frac{\lambda_{\theta}(s)}{A_{\theta}(\dot{s}(\theta),s)^2} + 1 \bigg\rbrace \label{eq: N dot s s definition}
\end{multline}
\end{lemma}
\begin{proof}
We compute that for any particular $\theta \in \mathcal{E}$,
 \begin{multline}
\lim_{\epsilon \to 0^+} \epsilon^{-1} \big( L_{\theta}(\dot{s}, s+\epsilon x ) - L_{\theta}(\dot{s}, s )  \big) = -\alpha x(\theta) \ell \bigg(  \frac{ \lambda_{\theta}(s)}{A_\theta(\dot{s},s)}\bigg) + D\lambda_{\theta}(s)\cdot x   \ell \bigg(  \frac{A_\theta(\dot{s},s)}{ \lambda_{\theta}(s)}  \bigg) \\
+ \alpha(1- s(\theta))  \log \bigg(  \frac{ \lambda_{\theta}(s)}{A_\theta(\dot{s},s)}\bigg) \bigg( \frac{A_\theta(\dot{s},s) D\lambda_{\theta}(s)\cdot x - \lambda_{\theta}(s) DA_{\theta}(\dot{s},s)\cdot x}{A_\theta(\dot{s},s)^2} \bigg) \\
+   \lambda_{\theta}(s) \log \bigg(  \frac{A_\theta(\dot{s},s)}{ \lambda_{\theta}(s)}  \bigg) \bigg(\frac{ \lambda_{\theta}(s) DA_{\theta}(\dot{s},s)\cdot x- A_{\theta}(\dot{s},s) D\lambda_{\theta}(s)\cdot x}{\lambda_{\theta}(s)^2} \bigg) \\
:= \mathcal{M}_{\theta}(\dot{s},s) D \lambda_{\theta}(s)\cdot x + x(\theta)  \mathcal{N}_{\theta}(\dot{s},s) ,
\end{multline}
where $\mathcal{M}_{\theta}(\dot{s},s) $ is defined in \eqref{eq: M dot s s definition} and $\mathcal{N}_{\theta}(\dot{s},s) $ is defined in \eqref{eq: N dot s s definition}.
\end{proof}
Differentiating, we find that
\begin{align}
\frac{d}{dt} \frac{\partial L_{\theta}}{\partial \dot{s}} (\dot{s}(t,\theta), s ) =&  \frac{\partial^2 L_{\theta}}{\partial \dot{s}^2} (\dot{s}(t,\theta), s ) \ddot{s}(t,\theta) +    \mathcal{O}_{\theta}(\dot{s},s) \text{ where } \\
\mathcal{O}_{\theta}: & \mathcal{C}(\mathcal{E}) \times \mathcal{C}(\mathcal{E}) \mapsto \mathbb{R} \text{ is such that }\\
\mathcal{O}_{\theta}(\dot{s},s) :=& \lim_{\epsilon\to 0^+} \epsilon^{-1} \bigg( \frac{\partial L_{\theta}}{\partial \dot{s}} (\dot{s}(t,\theta), s + \epsilon \dot{s} ) -\frac{\partial L_{\theta}}{\partial \dot{s}} (\dot{s}(t,\theta), s )  \bigg) .
\end{align}
It remains to find a convenient expression for $\mathcal{O}_{\theta}(\dot{s},s) $.
\begin{lemma}
\begin{multline*}
\mathcal{O}_{\theta}(\dot{s},s) =  \frac{\partial A_{\theta}}{\partial \dot{s}}(\dot{s},s)  \bigg\lbrace 1  + \frac{\alpha(1-s(\theta) \lambda_{\theta}(\dot{s},s))}{A_{\theta}(\dot{s},s)^2}\bigg\rbrace \bigg\lbrace \frac{\Delta A_{\theta}(\dot{s},s) }{A_{\theta}(\dot{s},s)} -  \frac{\Delta \lambda_{\theta}(\dot{s},s)}{\lambda_{\theta}(\dot{s},s)}  \bigg\rbrace \\
+ \alpha \dot{s} \log\bigg( \frac{\lambda_{\theta}(s)}{A_{\theta}(\dot{s},s)} \bigg)  \bigg\lbrace 1 + \frac{\alpha(1-s(\theta))\lambda_{\theta}(s)}{A_{\theta}(\dot{s},s)^2}\bigg\rbrace \bigg( \dot{s}(\theta)^2 + 4\alpha \lambda_{\theta}(s) (1-s(\theta)) \bigg)^{-3/2} \times \\ \bigg( (1-s(\theta)) \Delta \lambda_{\theta}(\dot{s},s) - \lambda_{\theta}(s) \dot{s}(\theta) \bigg)  \\
+\alpha \frac{\partial A_{\theta}}{\partial \dot{s}}(\dot{s},s)   \log\bigg( \frac{\lambda_{\theta}(s)}{A_{\theta}(\dot{s},s)} \bigg) \bigg( \frac{\dot{s}(\theta) \lambda_{\theta}(\dot{s},s)}{A_{\theta}(\dot{s},s)^2} + \frac{2 (1-s(\theta)) \lambda_{\theta}(\dot{s},s)}{A_{\theta}(\dot{s},s)^3} \Delta A_{\theta}(\dot{s},s) \\- \frac{(1-s(\theta))}{A_{\theta}(\dot{s},s)^2} \Delta \lambda_{\theta}(\dot{s},s) \bigg).
\end{multline*}
where
\begin{align}
\Delta \lambda_{\theta}(\dot{s},s) &= \dot{s}(\theta) \beta \int_{\mathcal{E}}\mathcal{J}(\theta,\tilde{\theta})(1- s(\tilde{\theta})) d\tilde{\theta} - \beta s(\theta) \int_{\mathcal{E}} \mathcal{J}(\theta,\tilde{\theta}) \dot{s}(\tilde{\theta}) d\tilde{\theta} \\
\Delta A_{\theta}(\dot{s},s) &= \alpha \big( \dot{s}^2 + 4\alpha \lambda_{\theta}(s) (1- s(\theta)) \big)^{-1/2}\big( - \lambda_{\theta}(s)\dot{s}(\theta) + (1-s(\theta)) \Delta \lambda_{\theta}(s)  \big)
\end{align}
\end{lemma}

\section{Acknowledgements}

This work was funded by NSF-DMS 2511615.

\end{document}